\documentclass{article}

\usepackage{amsmath}
\usepackage{amsthm}
\usepackage{amsfonts}
\usepackage{amssymb}
\usepackage{hyperref, cleveref}
\crefformat{section}{§#2#1#3}
\usepackage{graphicx}
\usepackage{fullpage}
\usepackage{todonotes}
\usepackage{ytableau}
\usepackage[british]{babel}
\usepackage[utf8]{inputenc}
\usepackage[all,cmtip]{xy}
\usepackage{hyperref}
\usepackage[normalem]{ulem}
\usepackage{mathtools}
\usepackage{xspace}
\usepackage{breqn}
\usepackage{mathrsfs}
\usepackage{comment}
\usepackage{enumitem}
\providecommand{\mb}[1]{\mathbb{ #1} }
\usepackage{titlesec}
\usepackage{xcolor}
\usepackage{tikz-cd}
\newtheorem{theorem}{Theorem}
\numberwithin{theorem}{section}

\newtheorem{corollary}[theorem]{Corollary}

\newtheorem{lemma}[theorem]{Lemma}

\newtheorem{proposition}[theorem]{Proposition}

\theoremstyle{definition}
\newtheorem{definition}[theorem]{Definition}
\newtheorem{remark}[theorem]{Remark}

\newcommand{\add}{\operatorname{add}}
\newcommand{\ve}{\varepsilon}
\newcommand{\PP}{\mathbb{P}}
\newcommand{\F}{\mathbb{F}}

\newcommand{\Z}{\mathbb{Z}}
\newcommand{\Q}{\mathbb{Q}}
\newcommand{\End}{\operatorname{End}}

\newcommand{\Tr}{\operatorname{Tr}}
\newcommand{\Id}{\operatorname{Id}}

\newcommand{\Frob}{\operatorname{Frob}}
\newcommand{\GL}{\operatorname{GL}}

\newcommand{\ord}{\operatorname{ord}}

\newcommand{\gr}{\operatorname{gr}}
\newcommand{\arr}{\operatorname{ar}}

\newcommand{\smr}{\operatorname{smr}}
\newcommand{\nsmr}{\operatorname{nsmr}}
\newcommand{\tors}{\operatorname{tors}}
\newcommand{\iso}{\operatorname{iso}}
\newcommand{\SL}{\operatorname{SL}}

\newcommand{\Gal}{\operatorname{Gal}}
\newcommand{\mult}{\operatorname{mult}}

\newcommand{\can}{\operatorname{can}}

\newcommand{\abs}[1]{\left|#1\right|}
\newcommand{\legendre}[2]{\ensuremath{\left( \frac{#1}{#2} \right) }}
\newcommand{\etale}{\'etal\@ifstar{\'e}{e\xspace}}

\newcommand\nnfootnote[1]{%
 \begin{NoHyper}
 \renewcommand\thefootnote{}\footnote{#1}%
 \addtocounter{footnote}{-1}%
 \end{NoHyper}
}

\begin{document}
\author{Stevan Gajović, Lazar Radi\v{c}evi\'{c}, Matteo Verzobio}

\title{The density of elliptic curves over $\Q_p$ with a rational 3-torsion point or a rational 3-isogeny}
\date{}
\maketitle
\begin{abstract}
We determine the probability that a random Weierstrass equation with coefficients in the $p$-adic integers defines an elliptic curve with a non-trivial $3$-torsion point, or with a degree $3$ isogeny, defined over the field of $p$-adic numbers. We determine these densities by calculating the corresponding $p$-adic volume integrals and analyzing certain modular curves. Additionally, we give an algorithm to compute the corresponding probability in the case of $\ell\neq p$-torsion for $\ell>3$ prime. Finally, we prove that, for an elliptic curve over $\F_p$ with {an $\F_p$-rational} point of order $p$, a random lift of such a curve to $\Q_p$ has a non-trivial $p$-torsion point defined over $\Q_p$ with the probability $1/p$.
\end{abstract}
\renewcommand{\thefootnote}{}
\nnfootnote{2020 \emph{Mathematics Subject Classification}: Primary 11G07; Secondary 14H52, 11S80}
\nnfootnote{\emph{Key words and phrases}: Elliptic curves, torsion points, isogenies, p-adic fields, elliptic curves over local fields, p-adic integration.}
\renewcommand{\thefootnote}{\arabic{footnote}}
\setcounter{footnote}{0}
\section{Introduction}
\subsection{Main results}

Let $p \geq 3$ be a prime number. In this paper, we determine the probability that a random Weierstrass equation $y^2+a_1xy+a_3y=x^3+a_2x^2+a_4x+a_6$, with coefficients in the $p$-adic integers $\Z_p$, defines an elliptic curve with a non-trivial $3$-torsion point or with a degree $3$ isogeny defined over the field $\Q_p$ of $p$-adic numbers.
More precisely, we consider the affine space $W = \Z_p^5$ of tuples $[a_1,a_2,a_3,a_4,a_6]$ of coefficients, equipped with the Haar measure $\mu_{\Z_p^5}$ normalized so that $\mu_{\Z_p^5}(W)=1$. Let $
W_{\tors}\subset W$ be the {subset}  of tuples for which the equation $y^2+a_1xy+a_3y=x^3+a_2x^2+a_4x+a_6$ defines an elliptic curve $E$ over $\Q_p$ with a non-trivial $\Q_p$-rational 3-torsion point. Our main result is the following theorem. 

\begin{theorem}\label{thm:main}
 Let $p\geq3$. We have
		{\[
		\mu_{\Z_p^5}(W_{\tors})=\begin{cases} 
			\frac{p^2 ( 3 p^6+4p^2-4p + 4)}{8 (p^8 + p^6 + p^4+p^2 + 1)} &\text{ if } p\equiv 1 \pmod 3, \\
			\frac{p^2}{2 (p^2 + p + 1)} &\text{ if } p\equiv 2 \pmod 3,\\
 \frac{3}{26} &\text{ if } p=3.
		\end{cases}
		\]}
\end{theorem}

We also compute the probability that a random Weierstrass equation defines an elliptic curve with a $\Q_p$-rational isogeny of degree $3$. Let $
W_{\iso}\subset \Z_p^5$ be the subset of tuples $[a_1,a_2,a_3,a_4,a_6]$ of coefficients which define an elliptic curve with a $\Q_p$-rational $3$-isogeny.
\begin{theorem}\label{thm:mainiso}
Let $p>3$.
We have
\[
		\mu_{\Z_p^5}(W_{\iso})=\begin{cases} 
			\frac{3p^4 + 3p^3 + 4p^2 + 4}{4(p^4 + p^3 + p^2 + p + 1)} \text{ if } p\equiv 1 \pmod 3, \\
 \frac{p^4 + p^3 + 2p^2 + 2}{2(p^4 + p^3 + p^2 + p + 1)} \text{ if } p\equiv 2 \pmod 3.
		\end{cases}
		\]
	\end{theorem}
   
We also consider {the more general probability} that an elliptic curve over $\Q_p$ has an $\ell$-torsion point defined over $\Q_p$. We single out here the important special case that an elliptic curve over $\Q_p$ has a $p$-torsion point. Because {the multiplication-by-$p$ isogeny} is not separable in characteristic $p$, it is a delicate question to decide if {an $\F_p$-rational $p$-torsion point} on the reduction mod $p$ lifts to {a $\Q_p$-rational $p$-torsion point}. 

An elliptic curve $E/\Q_p$ {is said to have} \textit{anomalous} ordinary reduction if it has a good ordinary reduction $\tilde{E}$ and an $\F_p$-rational $p$-torsion point. We say that a Weierstrass equation for such a curve $E$ is anomalous. Balakrishnan, {\c{C}}iperiani, Mazur, and Rubin in \cite[Remark 5.8]{balakrishnan2024shadow} observed that when the anomalous elliptic curves are ordered by conductor, numerical data suggest that the proportion of those with a $\Q_p$-rational $p$-torsion point tends to $1/p$ as we increase the conductor. The following theorem, which we prove using Serre-Tate theory, gives an explanation for this phenomenon.
\begin{theorem} \label{main:ptors}
 The probability that an anomalous Weierstrass equation $W\colon y^2=x^3+a_4x+a_6$ defines an elliptic curve with a $\Q_p$-rational $p$-torsion {point} is $1/p$. Consequently, when elliptic curves over $\Q$ with anomalous reduction modulo $p$ are ordered by naive height, asymptotically a proportion of $1/p$ of them will have a $\Q_p$-rational $p$-torsion point. 
\end{theorem}
See Proposition \ref{prob lift of p torsion} for a more precise statement. We remark that curves with this property are also important in elliptic curve cryptography, as observed by Smart \cite{Smart-discrete-log}.

\paragraph{Motivation.} A fundamental idea in number theory is to study Diophantine equations by reducing them modulo a prime $p$. Often, this simple approach is enough to show that a given equation has no integral or rational solutions. A natural question, then, is: for a fixed prime $p$, how frequently does this method successfully rule out solutions for a given class of equations? In this article, we focus on this question in the context of elliptic curves, specifically examining how often reduction modulo $p$ can detect the absence of rational $3$-torsion points. 

Intuitively, the Haar measure on $\Z_p$ is the probability measure that captures the notion that the probability that a random integer is divisible by $p$ should be $1/p$. For example, let $E/\Q$ be an elliptic curve with good reduction at $p \ne 3$, so that the reduction $\Tilde{E}$ is an elliptic curve over $\F_p$. Then any non-zero $3$-torsion point of $E$ defined over $\Q_p$ reduces to a non-zero $3$-torsion point of $\Tilde{E}$ defined over $\F_p$. By iterating this procedure for several primes $p$ we obtain an efficient algorithm to constrain the size of the rational $3$-torsion subgroup of $E$. When $E$ does not have good reduction at $p$, the situation becomes more complicated, requiring the study of congruences modulo higher powers {of} $p$. The {field of} $p$-adic numbers provides an efficient way to package together information modulo various powers of $p$. Theorem \ref{thm:main} can be viewed as the determination of the total number of congruences that characterize elliptic curves with a non-trivial $\Q_p$-rational $3$-torsion point. We also consider the following related questions.

\paragraph{$\ell$-torsion points for $\ell>3$.} We consider the probability of a non-trivial $\Q_p$-rational $\ell$-torsion point, for general {primes} $\ell$ and $p$. Many, but not all, features of the case $\ell=3$ also appear in this general setting, as we discuss in Section~\ref{sec:comm}. We obtain an asymptotic result (Theorem~\ref{asymptotic probability}) as ${p \to \infty}$. We also compute these probabilities for families of twists in \S\ref{subsec:twists-l-torsion}.

\paragraph{On $p$-torsion over $\Q_p$.}\label{subsec:l=p-results} We discuss the case $p=\ell$ in Section~\ref{p-torsion section}, proving Theorem \ref{main:ptors} and Theorem \ref{thm:main} in the tricky $p=3$ case. The main difference in the $p=\ell$ case is that the multiplication-by-$p$ isogeny is not separable in characteristic $p$. In particular, while by Hensel's lemma an $\ell$-torsion point in $\Tilde{E}[\ell](\F_p)$ always lifts to an $\ell$-torsion point in $E[\ell](\Q_p)$ (if $\ell\neq p$), this {is not true} for $p$-torsion points. In fact, we show (Proposition \ref{prob lift of p torsion}) that the probability that a non-trivial $p$-torsion $\F_p$-point on an elliptic curve with good ordinary reduction lifts to a non-trivial $p$-torsion $\Q_p$-point is $1/p$. The main idea of the proof is to show that the Serre-Tate canonical lift of $\tilde{E}$ is the unique lift modulo $p^2$ with this lifting property.

\paragraph{Counting elliptic curves over $\Q$ having a non-trivial rational 3-torsion point.} Using Theorem \ref{thm:main}, by a straightforward application of the large sieve, we obtain an upper bound for the number of elliptic curves, defined by a short Weierstrass equation, that {have} a non-trivial rational $3$-torsion point and whose coefficients lie in a box of arbitrary side lengths and center, see Proposition~\ref{sieve bound} and Corollary~\ref{cor:sieve-cases}. 
The bound in Corollary~\ref{cor:sieve-cases}(1) remains larger in order of magnitude than what one might expect based on the results of Harron and Snowden~\cite{harron2017counting} (see also~\cite{cullinan2022probabilistic}), where they count elliptic curves of bounded height having a rational $3$-torsion point using geometry of numbers.
However, the flexibility afforded by our results leads to a new statement when counting such elliptic curves in so-called short intervals; see Corollary~\ref{cor:sieve-cases}(2). The significance of taking
short intervals lies in the fact that statistical patterns start to emerge
even when sampling a very thin subset of elliptic curves. These intervals have become of major interest recently, especially after the breakthrough paper \cite{matomaki2016multiplicative} by Matom{\"a}ki and Radziwi{\l}{\l}. 

\subsection{Previous work on related topics}

\paragraph{$p$-adic probabilities in other contexts.} While a priori the $p$-adic density of an arbitrary subset of $\Z_p^5$ does not even need to be a rational number, note that the densities $\mu_{\Z_p^5}(W_{\tors})$ and $\mu_{\Z_p^5}(W_{\iso})$ are fixed rational functions of $p$ (depending only on the class of $p$ mod 3). This happens frequently in
the study of densities of $\Q_p$-soluble equations for many other classes of equations. For example, in \cite{BCFG}, Bhargava, Cremona, Fisher, and the first author show that the probability that a random degree $n$ polynomial in $\Z_p[x]$ has a rational zero is a rational function of $p$. Other cases that were investigated are {the density of isotropic
quadratic forms over a fixed number of variables over $\Z_p$ by Bhargava, Cremona, Fisher, Jones, and Keating \cite{BCFJK}},
 $\Q_p$-solubility of genus one curves of degree $2$ and $3$ \cite{bhargava2021proportion}, \cite{bhargava2016proportion} by Bhargava, Cremona, and Fisher, $(2,2)$-forms \cite{fisher2021everywhere} by Fisher, Ho, and Park, and cubic hypersurfaces \cite{Beneish-Keyes-cubic-hypersurface} by Beneish and Keyes. In all these cases the corresponding probability was computed to be the value of a certain rational function of $p$. 

{Remarkably, this rational function is independent of the prime $p$ in almost all of these results, except for a small complication in the case $p=2$ for genus one curves of degree two (see \cite{bhargava2021proportion}), where it is only independent of $p$ when we parametrize these curves using generalized quartic models $y^2+hy=g$.} 

\paragraph{The case of 2-torsion.} Note that in the case $\ell=2$, the probability that $E \colon y^2=f(x)$ has a non-trivial $\Q_p$-rational $2$-torsion point is the probability that $f(x)$ has a $\Q_p$-rational root. This probability can be readily derived from \cite[Section 1.2.3]{BCFG}, where the probability of a random cubic polynomial possessing a $\Q_p$-rational root is computed. This is based on the observation that $f(x)$ is essentially the 2-division polynomial of $E$. This strategy however does not extend well to the case $\ell>2$, since then the $\ell$-division polynomials of elliptic curves form a very special family of degree ${(\ell^2-1)/2}$ polynomials in $\Z_p[x]$ which are not well-suited for the recursive method of \cite{BCFG}. 

\paragraph{Asymptotic counting over $\Q$.} Recently, the study of the asymptotic counts of elliptic curves over $\Q$ with a $\Q$-rational non-trivial $\ell$-torsion point has also attracted a lot of attention{;} see \cite{harron2017counting}, \cite{cullinan2022probabilistic}, \cite{pizzo2020counting}, and \cite{bruin2022counting} among others. In particular, Harron and Snowden \cite[Theorem 1.3]{harron2017counting} provide an asymptotic formula for the number of elliptic curves of bounded height with a $\Q$-rational non-trivial $3$-torsion point. An elliptic curve over $\Q$ with a $\Q$-rational non-trivial $3$-torsion point must also have a $\Q_p$-rational non-trivial $3$-torsion point for all primes $p$, and in Section~\ref{sieve section} we prove an upper bound for the number of elliptic curves by imposing local conditions at many primes 
$p$ and sieving.

\subsection{Strategy and organization of the paper}

\paragraph{Difficulties and challenges.}\label{subsec:difficulties} Previous work on determining densities of $p$-adically soluble equations in families was, as far as we are aware, restricted to families of hypersurfaces parametrized by an affine space, and the densities were determined by an explicit analysis of the equation that defines the hypersurface via a recursive method. In our setting, we consider the family of 3-torsion subgroups $E_{a,b}[3]$, where $E_{a,b} \colon y^2=x^3+ax+b$, as the pair $(a,b)$ varies in the affine plane $\Z_p^2$. As a variety, $E_{a,b}[3]$ consists of 9 geometric points, one of which corresponds to the origin, and we want to understand how often one of the remaining 8 points is rational. There appears to be no natural way to view $E_{a,b}[3]$ as a hypersurface, and its defining equations are rather complicated, involving the 3-division polynomial and the defining equation of $E_{a,b}$. This is a significant difficulty for the recursive method, which we circumvent by making use of $p$-adic integration and the theory of modular curves. 

\paragraph{Modular curves.} Throughout the paper, we exploit the fact that the modular curves $X_1(3)$ and $X(3)$ that parametrize elliptic curves with level-3 structure are isomorphic to the projective line $\PP^1$. This is essentially the classical theory of Tate and Hesse normal forms of elliptic curves. We review this theory in Section \ref{moduli space section}. We prove Theorem \ref{thm:main}, in the case $p>3$, in Section \ref{sec:3tors}. The proof splits into cases according to the reduction type of the Weierstrass equations. For Weierstrass equations with good reduction, we count them by counting $\F_p$-points on $X_1(3)$ and $X(3)$. For Weierstrass equations with bad reduction, their $3$-torsion subgroup can be determined solely from their Kodaira type. To compute the probability in this case, we rely on the results of Cremona and Sadek \cite{cremona2020localpub}, which provide the probabilities of an elliptic curve having a given Kodaira type. 

\paragraph{$p$-adic integration.} In Section \ref{sec:iso} we prove Theorem \ref{thm:mainiso} and give a second proof of Theorem \ref{thm:main}. Admitting a rational $3$-isogeny is equivalent to the $3$-division polynomial $\psi_3$ (see \eqref{eq:psi3}) having a rational root. One might attempt to find such a root modulo $p$ and then lift it via Hensel's lemma. This is hard to perform when the reduction of $E$ modulo $p$ is not smooth, since many different cases must be dealt with individually. For this reason, we carry out the counting through $p$-adic integration with respect to a Haar measure on $\Z_p^2$. The main tool is the $p$-adic change of variables formula, which allows us to compute $p$-adic volumes by computing integrals of $p$-adic absolute values of Jacobian determinants. The $p$-adic integral can be viewed as an extension of the $\F_p$-point count on the modular curve to the points of bad reduction, and allows us to treat the cases of good and bad reduction uniformly. The method is very flexible and can be used to compute $p$-adic densities of many sets of interest in arithmetic statistics, arising in a variety of situations. For example, it is also used in the work of Bhargava \cite[Theorem 21]{bhargava2008higher} on counting quintic number fields, in the proof of Serre's mass formula \cite{serre1978formule} for totally ramified extensions of $\Q_p$ and in Tate's treatment of the functional equation of the Riemann zeta function. 

\paragraph{Final sections{.}}

As already mentioned, in Section~\ref{p-torsion section} we study the case of $p$-torsion in $\Q_p$. We study the case $p\neq \ell$ in Section \ref{subsec:l-torsion}. For $\ell=2,3,4,5$, the modular curve $X(\ell)$ that parametrizes elliptic curves with full level-$\ell$ structure is isomorphic to $\PP^1$, and so we expect that Theorem \ref{thm:main} generalizes to these $\ell$ without any major changes. For $\ell>5$, the curve $X(\ell)$ has positive genus, and the number of its $\F_p$-points is typically not a fixed polynomial of $p$. In Remark \ref{rmk:l-torsion} we give a formula for the probability as a function of $p$ and these $\F_p$-point counts, showing in particular that it is a rational number. We do not expect that the probability is a fixed rational function of the variable $p$ alone for any $\ell>5$. We, however, have a simple asymptotic formula (Theorem \ref{asymptotic probability}) for this probability in the limit $p \to \infty$.

We did not analyze the case $p=2$ in Theorem \ref{thm:main}, nor the cases $p=2,3$ in Theorem \ref{thm:mainiso}. However, with the methods developed in this paper, in particular in Section \ref{p-torsion section}, these cases should be straightforward to handle, albeit somewhat tedious. Since their inclusion would not add much to the mathematical content of this work, we have chosen to omit them.

\paragraph{Acknowledgements{.}} We thank Jennifer Balakrishnan, John Cremona, Netan Dogra, Antonio Lerario, Damien Robert, Damaris Schindler, and Jaap Top for helpful conversations and comments. We are very grateful to Efthymios Sofos for his interest in this work and some very useful comments on the first draft of this paper. In particular, his comments led to an improvement of Section \ref{sieve section}. {Finally, 
we would like to thank the anonymous referee for the careful reading of the manuscript, numerous corrections, several useful suggestions, questions, and several improvements from which this article benefited a lot.} The authors want to thank Mathematical Institute of Serbian Academy of Sciences, Charles University, ISTA, and University of Groningen for their hospitality when they collaborated. The first author was supported by {the Spanish Ministry of Science and Innovation / State Research Agency (Ministerio de Ciencia e Innovación / Agencia Estatal de Investigación), through the Ramón y Cajal program, grant number RYC24SG-AD}, the MPIM guest postdoctoral fellowship program, Czech Science Foundation GAČR, grant 21-00420M, Junior Fund grant for postdoctoral positions at Charles University, and Mathematical Institute of Serbian Academy of Sciences during various stages of this project. The second author was supported by the Max Planck Institute for Mathematics, Mathematical Institute of Serbian Academy of Sciences and by the UKRI grant MR/T041609/2 during his stay at King's College London. While working on this paper, the third author was supported by the European Union’s Horizon 2020 Research and Innovation Programme under the Marie Skłodowska-Curie Grant Agreement No. 101034413 and by the Research Foundation – Flanders (FWO) through grants G0B3123N, G0B1721N, and G0A2923N at KU Leuven.

\section{Preliminaries} \label{preliminaries} 
 We are going to recall some basic facts on elliptic curves that we are going to use in the paper. For more details and all the proofs, see \cite{silverman2009arithmetic}. In particular, for the properties of elliptic curves over local fields, we always refer to \cite[Chapter 7]{silverman2009arithmetic}.

When $p>3$, the change of variables that reduces a long Weierstrass equation to a short one with $a_1=a_2=a_3=0$ is measure preserving (see Proposition ~\ref{prop:assume-short-WE}) and so throughout the paper (except in the case $p=3$) we will restrict attention to the Weierstrass equations of the form $y^2=x^3+a_4x+a_6$, where $a_4,a_6\in \Z_p$.
 
 Let $E$ be an elliptic curve defined over a field $K$. If $\ell$ is a prime and ${\zeta_\ell}$, a primitive $\ell$-th root of unity, is not in $K$, then $E[\ell](K)$ cannot have order $\ell^2$. For all $n\geq 1$, we can find two polynomials $\theta_n(x,y)$ and $\psi_n(x,y)$, depending on the Weierstrass equation defining $E$, such that for non-zero $P\in E(K)$ with $nP\neq 0_E$ we have
 \[
 x(nP)=\frac{\theta_n(x(P),y(P))}{\psi_n^2(x(P),y(P))}
 \]
 and $\psi_n^2(x(P),y(P))=0$ if and only if $nP=0_E$.
 In particular, if $E$ is defined by the Weierstrass equation $y^2+a_1xy+a_3y=x^3+a_2x^2+a_4x+a_6$, then $\psi_3$ is the 3-division polynomial: 
 \begin{equation}\label{eq:psi3}
 \psi_3\left(x,y\right)= 3x^4 + \left(a_1^2 + {4a_2}\right)x^3 + 3\left(2a_4+a_1a_3\right)x^2 + 3\left(a_3^2+4a_6\right)x + a_1^2a_6 + 4a_2a_6 -a_1a_3a_4 + a_2a_3^2 -a_4^2.
 \end{equation}

 Let $E$ be an elliptic curve defined over $\Q_p$ by a Weierstrass equation $y^2+a_1xy+a_3y=x^3+a_2x^2+a_4x+a_6$. All the isomorphisms from $E$ to another Weierstrass equation are given by maps of the form $(x,y)\mapsto (u^2x+r,u^3y+u^2sx+t)$, where $u \in {\overline{\Q_p}^*}$ and ${r,s,t \in \overline{\Q_p}}$ which we call a change of variables. If $E$ is in short Weierstrass form, that is defined by $y^2=x^3+Ax+B$, then {the changes of variables preserving this form are the maps} $(x,y)\mapsto (u^2x,u^3y)${; each gives a map} $E\to E_u$ with $E_u$ defined by $y^2=x^3+Au^4x+Bu^6$. Notice in particular that $E_u$ is defined over $\Q_p$ if and only if $u^2\in\Q_p$ (unless $A=0$ or $B=0$).

 For the definition of elliptic curves in minimal form, see \cite[Section VII.1]{silverman2009arithmetic}{.} For the definition and properties of good, additive, and multiplicative reduction, see \cite[Section VII.5]{silverman2009arithmetic}. For the Kodaira symbols, we refer to \cite[Table C.15.1]{silverman2009arithmetic}. We will often use Hensel's lemma, as stated in \cite[Lemma IV.1.2]{silverman2009arithmetic}, and see \cite[Corollary 6.2.13]{qingliu} for the geometric version.

\section{Moduli spaces of elliptic curves with level-3 structure} \label{moduli space section}
In this section, we review some well-known results on parametrizations of moduli spaces of Weierstrass equations that define elliptic curves with level-3 structure. Classically these results are stated in terms of modular curves, which parametrize elliptic curves with level $3$-structure up to isomorphism. Note that a point on a modular curve only determines an elliptic curve up to isomorphism and does not specify a preferred Weierstrass equation for the curve. We restate these modular parametrizations in terms of a finer moduli space that parametrizes Weierstrass equations of elliptic curves, which allows us to treat $j$-invariants $0$ and $1728$ uniformly and simplifies our counting problem. Throughout this section we will assume $k$ is a perfect field of characteristic $\ne 2,3$.
\subsection{Counting Weierstrass equations} 
We first observe that counting short Weierstrass equations is the same as counting elliptic curves with a regular differential. For {the} definition and properties of differentials, see \cite[Section III.5]{silverman2009arithmetic}. Consider pairs $(E,\omega)$ of an elliptic curve $E/k$ and a {non-zero} regular differential $1$-form $\omega$ on $E$. We say two pairs $(E_1,\omega_1)$ and $(E_2,\omega_2)$ are $k$-isomorphic if there is a $k$-isomorphism $\gamma\colon E_1 \xrightarrow[]{} E_2$ with $\gamma^*\omega_2=\omega_1$. 
\begin{definition}
 Define $X_{\omega}(k)$ as the set of $k$-isomorphism classes $(E,\omega)$.
\end{definition}
\begin{proposition}\label{lem:Xomega}
There is a bijection between the set ${\mb{A}^2_{X}(k)}=\{(a,b) \in {k^2}: 4a^3+27b^2 \ne 0\}$ and the set $X_{\omega}(k)$ given by the map $\phi \colon \mb{A}^2_{X}(k) \xrightarrow[]{} X_{\omega}(k) $ defined as
\[
(a,b) \mapsto \left(E_{a,b}\colon y^2=x^3+ax+b, \frac{dx}{y}\right).
\]
\end{proposition}
\begin{proof}
We check injectivity: If $(E_{a',b'},\frac{dx'}{y'}) \cong (E_{a,b},\frac{dx}{y})$, then there is a change of variables $(x',y')=\gamma(x,y)=(u^2x,u^3y)$ and then $a'=u^4 a$ and $b'=u^6 b$. We compute $\frac{dx}{y}=\gamma^*(\frac{dx'}{y'})=u^{-1}\frac{dx}{y}$, so $u=1$. We now check surjectivity. Given a pair $(E,\omega) \in X_{\omega}(k)$, choose a Weierstrass model $E_{a,b}$ for $E$. Since the space $H^0(E,\Omega_E)$ of global differentials on $E$ is $1$-dimensional by the Riemann-Roch theorem, we can achieve $\frac{dx}{y}=\omega$ by rescaling the pair $(a,b)$.
\end{proof}

\subsection{Elliptic curves with a marked \texorpdfstring{$3$}{}-torsion point} \label{marked torsion point section}
 We recall the well-known parametrization of elliptic curves with a marked $k$-rational non-trivial $3$-torsion point.
\begin{lemma} \label{tate normal form}
	Let $E/k$ be an elliptic curve, $P\in E[3](k)$ a non-trivial $3$-torsion point. Then there is an isomorphism $i \colon E \xrightarrow{} E'$ defined over $k$ where $E'$ is given by
	\[
	E'\colon y^2+u xy+v y=x^3
	\]
	with $i(P)=(0,0)$. The pairs $(u,v)$ are unique up to scaling by $s \in k^{*}$: $(u,v) \mapsto (su,s^3v)$, and we have a bijection between pairs $(u,v)$ up to this equivalence relation and isomorphism classes of pairs $(E,P)$.
\end{lemma}
\begin{remark}
 The equation $ y^2+u xy+v y=x^3$ is known as the Tate normal form of an elliptic curve.
\end{remark}
\begin{proof}
Suppose $E$ is given by
	\[
	E\colon y^2+a_1xy+a_3y=x^3+a_2x^2+a_4x+a_6.
	\]
	Choose coordinates so that $P=(0,0)$ and $y=0$ is the tangent line at $P$, so that $a_4=a_6=0$. {As $P$ is an inflection point, it follows that $a_2=0$}. Any two different Weierstrass equations corresponding to $(u_1,v_1)$ and $(u_2,v_2)$ are related by a substitution 
 $x'=(u')^2 x+r', y'=(u')^3y+(u')^2s'x+t'$, 
 and since in both sets of coordinates $P=(0,0)$ and $y=0$ is the tangent line at $P$, we have $r'=s'=t'=0$ and $(u_2,v_2)=(u'u_1,(u')^3v_1)$ and the claim follows.
\end{proof}

We consider triples $(E,\omega,P)$, where $\omega$ is a regular {non-zero} 1-form on $E$ and we have $P \in E[3](k)$. Two triples $(E_1,\omega_1,P_1)$ and $(E_2,\omega_2,P_2)$ are $k$-isomorphic if there is {a $k$-isomorphism} $\gamma \colon E_{1} \xrightarrow[]{} E_2$ with $\gamma^* \omega_2=\omega_1$ and $\gamma(P_1)=P_2$. 
\begin{definition}
 Let $X_{\omega,1}(3)(k)$ denote the set of $k$-isomorphism classes $(E,\omega,P)$, where $\omega$ is a regular {non-zero} 1-form on $E/k$ and $P \in E[3](k) {\setminus \{0_E\}}$.
\end{definition}
 Let 
 \begin{equation}\label{eq:A21}
 \mb{A}^2_{X_1(3)}=\left\{(u,v) \in \mb{A}^2 : \left(u^3-27v\right)v^3 \neq 0 \right\}.
 \end{equation}
Notice that $(u^3-27v)v^3$ is the discriminant of the Weierstrass equation in Tate normal form.

\begin{proposition}\label{marked 3-tors point parametrization prop}
	There is a bijection $\phi_1 \colon \mb{A}^2_{X_1(3)}(k) \xrightarrow{} X_{\omega,1}(3)(k) $ defined by 
	\[
	(u,v) \mapsto \left(y^2+uxy+vy=x^3, \frac{dx}{{2}y+ux+v},(0,0)\right).
	\]		
\end{proposition} 
\begin{proof}
We observe that rescaling the pair $(u,v) \mapsto (su,s^3v)$ in Tate normal form rescales the differential $\omega \mapsto s^{-1} \omega$, and the result follows immediately from the proof of the previous lemma. 
\end{proof}
To count isogenies of degree $3$ we will also need a twisted variant of $X_{\omega,1}(3)$. First, we recall a simple criterion for an elliptic curve to have a $k$-rational $3$-isogeny.
\begin{lemma} \label{3-isogeny lemma}
 An elliptic curve $E/k \colon y^2=f(x)$ has a $k$-rational $3$-isogeny if and only if it admits a Weierstrass equation of the form $y^2=(x-\alpha)^3+\lambda(u(x-\alpha)+v)^2$ where $\lambda\in k^{*}$ and $\alpha,u,v\in k$. The kernel of the $3$-isogeny is $\{0_E,{(\alpha, \pm\sqrt{\lambda}v)}\}$.
\end{lemma}
\begin{proof}
 Suppose that $E$ has a $k$-rational $3$-isogeny with kernel generated by a point $P=(x(P),y(P))$. For all $\sigma\in \Gal(\overline{k}/k)$, we have $P^\sigma=\pm P$, so that we have $x(P)\in k$ and $y(P)$ is defined over {an (at worst) quadratic extension} $k(\sqrt{\lambda})/k$, for some $\lambda \in k^{*}$ (if $y(P)\in k$, we put $\lambda=1$). As the divisor ${3 \cdot P-3 \cdot 0_E}$ is principal, there exists a rational function on $E$ which has a triple pole at $0_E$, a triple zero at $P$ and no other zeros or poles, which we can write as $y+Ax+B$, for some $A,B \in k(\sqrt{\lambda})$. We observe that $y-Ax-B$ has a triple zero at $-P$ and no other zeros, and hence we have $(y+Ax+B)(y-Ax-B)=(x-x(P))^3$, and so $y^2=(x-x(P))^3+(Ax+B)^2$. If $P \not\in E(k)$, then $P$ and $-P$ are Galois conjugate, so we must have $x(P), A/\sqrt{\lambda},B/\sqrt{\lambda} \in k$, and our claim follows. Conversely, from the factorization 
 \[
 \left(y-\sqrt{\lambda}(u(x-\alpha)+v)\right)\left(y+\sqrt{\lambda}(u(x-\alpha)+v)\right)=(x-\alpha)^3
 \]
 we see that the divisor of $(y-\sqrt{\lambda}(u(x-\alpha)+v))$ is ${3\cdot(\alpha,\sqrt{\lambda}v)-3\cdot 0_E}$, and so $P=(\alpha,\sqrt{\lambda}v)$ is a $3$-torsion point defined over $k(\sqrt{\lambda})$, and so generates a $k$-rational $3$-isogeny.
\end{proof}

\begin{remark}
 When $\lambda=1$, the above normal form can also be obtained by completing the square on the left-hand side of the Tate normal form of the elliptic curve.
\end{remark}

Fix $\lambda \in k^{*}\setminus (k^{*})^2$, and let $\psi \colon \mathrm{Gal}(\overline{k}/k) \xrightarrow{} \{\pm1\}$ be the {quadratic} character associated with the quadratic field extension $k(\sqrt{\lambda})$, defined by $\psi(\sigma)=\frac{\sigma(\sqrt{\lambda})}{\sqrt{\lambda}}$ for $\sigma \in \mathrm{Gal}(\overline{k}/k)$. 

 Consider triples $(E,\omega,P)$, where $\omega$ is a regular {non-zero} 1-form on $E$ and $P \in E[3](\overline{k})$ is such that we have $\sigma(P)=\psi(\sigma)P$ for all $\sigma \in \mathrm{Gal}(\overline{k}/k)$. Equivalently, if $\psi$ is non-trivial, $ P\in E[3](k(\sqrt{\lambda}))$ and $ P \not\in E[3](k)$. Then $\{0_E,P,-P\}$ is a cyclic subgroup of $E$, defined over $k$, on which $\mathrm{Gal}(\overline{k}/k)$ acts through the character $\psi$. We say that two triples $(E_1,\omega_1,P_1)$ and $(E_2,\omega_2,P_2)$ are $k$-isomorphic if there is a $k$-isomorphism $\gamma \colon E_{1} \xrightarrow[]{} E_2$ with $\gamma^* \omega_2=\omega_1$ and $\gamma(P_1)=P_2$. 
 \begin{definition} 
 Let $X_{\omega,\psi}(3)(k)$ denote the set of $k$-isomorphism classes of $(E,\omega,P)$ where $E/k$ is an elliptic curve, $\omega$ is a regular {non-zero} 1-form on $E$ and $P \in E[3](\overline{k})$ is a non-zero point such that we have $\sigma(P)=\psi(\sigma)P$ for all $\sigma \in \mathrm{Gal}(\overline{k}/k)$.
 \end{definition}
Define 
\[
\mb{A}^2_{X_{\psi}(3)} = \left\{ (u,v) \in \mb{A}^2 : 4\lambda^3u^3v^3 - 27\lambda^2v^4 \ne 0 \right\}.
\] 
Notice that $16(4\lambda^3u^3v^3 - 27\lambda^2v^4)$ is the discriminant of the elliptic curve defined in the statement of Lemma \ref{3-isogeny lemma}.

\begin{proposition} \label{3-isogeny parametrization}
 There is a bijection $\phi_{\psi} \colon \mb{A}^2_{X_{\psi}(3)}(k) \xrightarrow{} X_{\omega,\psi}(3)(k)$ defined by
 \[
 (u,v) \mapsto \left(y^2=x^3+\left(2\lambda u v-\frac{\lambda^2u^4}{3}\right)x+\left(\lambda v^2-\frac{2\lambda^2vu^3}{3}+\frac{2\lambda ^3u^6}{27}\right), \frac{dx}{y}, \left(\frac{\lambda u^2}{3},\sqrt{\lambda}v\right)\right).
 \]

\end{proposition}
\begin{proof}
 The unique value of $\alpha$ for which the equation of Lemma \ref{3-isogeny lemma} is a short Weierstrass equation is $\alpha=\frac{\lambda u^2}{3}$. The statement now follows from the proof of Lemma \ref{3-isogeny lemma}, putting $\alpha=\frac{\lambda u^2}{3}$.
\end{proof}

\subsection{Elliptic curves with full level-3 structure} \label{Full level-3 section}
Now, we describe the family of elliptic curves with $E[3]$ that has a certain Galois structure. Consider the $\mathrm{Gal}(\overline{k}/k)$-module $M\coloneq\Z/3\Z \times \mu_3$ equipped with the alternating pairing {$\wedge^2 M \xrightarrow{} \mu_3$} defined by $((a,\zeta_3^{b}),(c,\zeta_3^{d}))=\zeta_3^{ad-bc}$, where $\zeta_3$ is a primitive cube root of unity and $\mu_3=\langle \zeta_3\rangle$. We consider pairs $(E,\phi)$, where $E$ is an elliptic curve over $k$ and $\phi \colon M \xrightarrow[]{} E[3](\overline{k})$ is an isomorphism defined over $k$, which is compatible with the alternating pairing on $M$ and the Weil pairing on $E[3](\overline{k})$ (notice in particular that $E$ has a $k$-rational non-trivial $3$-torsion point). We say that two pairs $(E,\phi)$ and $(E',\phi')$ are isomorphic if there is a $k$-isomorphism $\psi \colon E \xrightarrow{} E'$ such that $\phi'=\psi \circ \phi$. Pairs $(E,\phi)$ are parametrized by an open subset of the modular curve $X(3) \cong \PP^1$, and a point $(u:v) \in \PP^1(k)$ corresponds to the elliptic curve 
\[
C_{u,v}\colon u(x^3+y^3+z^3)-3vxyz=0
\]
with $(0:1:-1)$ as the identity element, and the isomorphism $\phi \colon M \xrightarrow[]{} C_{u,v}[3](\overline{k})$ defined by {$\phi(1,1)=(1:-1:0)$} and  {$\phi(0,\zeta_3) =(0:\zeta_3^2:-\zeta_3)$}. This family of plane cubic curves is known as the Hesse pencil, see \cite[Section 1.1]{rubin1995families} or \cite{artebani2009hesse}.

As before, to count Weierstrass equations, we add the data of a regular differential $\omega$ on $E$ to a pair $(E,\phi)$ and consider triples $(E,\omega,\phi)$. We define the notion of $k$-isomorphism of such triples in a similar way as before.

\begin{definition}
 Define $X_{\omega}(3)(k)$ to be the set of $k$-isomorphism classes {of} triples $(E,\omega,\phi)$ where $E/k$ is an elliptic curve, $\omega$ is a regular {non-zero} differential on $E$, and $\phi \colon M \xrightarrow[]{} E[3](\overline{k})$ is an isomorphism defined over $k$, which is compatible with the alternating pairing on $M$ and the Weil pairing on $E[3](\overline{k})$.
\end{definition} 
Let 
\begin{equation}\label{eq:A2}
\mb{A}^2_{X(3)}=\left\{(u,v) \in \mb{A}^2 : u\left(u^3-v^3\right) \ne 0\right\}.
\end{equation}
{$C_{u,v}$ is nonsingular if and only if $u(u^3-v^3)\ne0$.}
\begin{proposition} \label{split X(3)}
	There is a bijection $\phi_2 \colon \mb{A}^2_{X(3)}(k) \xrightarrow[]{} X_{\omega}(3)(k)$ defined by
	\[
	(u,v) \mapsto \left( u(x^3+y^3+z^3)-3vxyz=0,\frac{
 y dz -zdy}{3ux^2-3vyz},\phi\right) 
	\]
	where the isomorphism $\phi \colon \Z/3\Z \times \mu_3 \xrightarrow[]{} E[3](\overline{k})$ sends the generators $(0,\zeta_3)$ and $(1,1)$ of $\Z/3\Z \times \mu_3$ to $\phi (0,\zeta_3)={(0:\zeta_3^2:-\zeta_3)}$ and $\phi(1,1)=(1:-1:0)$. 
\end{proposition}
\begin{proof}
 The differential $\frac{
 y dz -zdy}{3ux^2-3vyz}$ is a regular differential on $C_{u,v}$ by the adjunction formula applied to $C_{u,v} \subset \PP^2$ (see also \cite[Lemma 5.21]{g1inv}), and so the map $\phi_2$ is well-defined. Rescaling the pair $(u,v)$ by a $\lambda \in k^{*}$ rescales $\frac{
 y dz -zdy}{3ux^2-3vyz}$ by $\lambda^{-1}$. The family $C_{u,v}$ is the universal elliptic curve with level-3 structure over the open modular curve $Y(3) \cong \{(u:v) \in \PP^1 : u(u^3-v^3) \ne 0\}$, see for example \cite[Section 1.1]{rubin1995families} for a proof. Thus any pair $(E,\phi)$ arises from a pair $(u,v) \in \mb{A}^2_{X(3)}$ that is unique up to scaling, and fixing a differential $\omega$ on $E$ pins down the scaling uniquely. 
\end{proof}
\begin{corollary}\label{cor:bij3structure}
	There is a bijection $\phi_2 \colon \mb{A}^2_{X(3)}(k) \xrightarrow[]{} X_{\omega}(3)(k)$ defined by
	\[
	(u,v) \mapsto \left(y^2 = x^3 + \left(-\frac{27}{2}u^3v - \frac{27}{16}v^4\right)x + \left(-\frac{27}{4}u^6 - \frac{135}{8}u^3v^3 + \frac{27}{32}v^6\right),\frac{dx}{{2}y},\phi\right)
	\]
	where the isomorphism $\phi \colon \Z/3\Z \times \mu_3 \xrightarrow[]{} E[3](\overline{k})$ sends the generators $(1,1)$ and $(0,\zeta_3)$ of $\Z/3\Z \times \mu_3$ to 
 \[
 \phi(1,1)=\left(\frac 34\left(4u^2 + 4uv + v^2\right), -\frac 92\left(u^3 +u^2v +uv^2\right) \right)\] and \[\phi (0,\zeta_3)=\left(-\frac 94 v^2, \frac{3}{2}\left(2\zeta_3 + 1\right)\left(u^3-v^3\right)\right).
 \]
\end{corollary}
\begin{proof}
We compute a Weierstrass equation for the curve $C_{u,v}$ by using Weil's formula for the Weierstrass equation of {a} plane cubic, see \cite[Theorem 3.1]{an2001jacobians} and \cite[Theorem 1.2]{fisher2018invforalln} for a more general version. It is given by \[ 
E_{u,v}\colon y^2 = x^3 + \left(-\frac{27}{2}u^3v - \frac{27}{16}v^4\right)x + \left(-\frac{27}{4}u^6 - \frac{135}{8}u^3v^3 + \frac{27}{32}v^6\right).
\]
For the explicit Weierstrass equation, see \cite[Equation 1]{rubin1995families}.
In fact, \cite[Theorem 1.2]{fisher2018invforalln} implies that we have an isomorphism of pairs $(E,\omega)$: 
\begin{align*}&\left(y^2 = x^3 + \left(-\frac{27}{2}u^3v - \frac{27}{16}v^4\right)x + \left(-\frac{27}{4}u^6 - \frac{135}{8}u^3v^3 + \frac{27}{32}v^6\right),\frac{dx}{{2}y}\right) \\\cong &\left( u\left(x^3+y^3+z^3\right)-3vxyz=0,\frac{
 y dz -zdy}{3ux^2-3vyz}\right).\end{align*}
A direct calculation shows that $P_1=\phi(1,1)$ and $P_2=\phi(0,\zeta_3)$ are $3$-torsion points with Weil pairing $e(P_1,P_2)=\zeta_3$, and that $\phi \colon {M} \xrightarrow{} E_{u,v}[3](\overline{k})$ is a Galois equivariant isomorphism.
\end{proof}

We will also need a twisted form of this parametrization. Fix an elliptic curve $F/k \colon y^2=x^3+ax+b$ and consider the modular curve $X_F(3)$, the {noncuspidal} $k$-points of which parametrize elliptic curves $E/k$ with an isomorphism $\phi \colon F[3](\overline{k}) \xrightarrow{} E[3](\overline{k})$ that respects the Weil pairing. The curve $X_F(3)$ is a twist of $X(3) \cong \PP^1$ and is also isomorphic to $\PP^1$, since there is a $k$-rational point on $X_F(3)$ corresponding to the curve $F$ itself. Theorem 13.2 of \cite{fisher2012hessian} gives the following explicit parametrization: a {noncuspidal} point $(u: v) \in X_F(3)(k)$ determines the elliptic curve
\[
E\colon y^2=x^3+\mathbf{c}_4(u,v)x+\mathbf{c}_6(u,v)
\]
and an isomorphism $\phi_{u,v} \colon F[3](\overline{k}) \xrightarrow[]{} E[3](\overline{k})$ of Galois modules. The coefficients $\mathbf{c}_4,\mathbf{c}_6$ are called \textit{Hesse} polynomials, given by 
\[
\mathbf{c}_4=-6912 a^3 v^4 - 288 a^2 u^2 v^2 - 3456 a b u v^3 + a u^4 - 62208 b^2 v^4 + 72 b u^3 v
\]
and
	\begin{align*}
		\mathbf{c}_6\coloneq&-110592 a^4 u v^5 - 995328 a^3 b v^6 + 34560 a^2 b u^2 v^4 - 16 a^2 u^5 v\\& - 497664 a b^2 u v^5 - 720 a b u^4 v^2 - 5971968 b^3 v^6 - 17280 b^2 u^3 v^3 + b u^6.
	\end{align*}
We define $\mathbf{disc}(u,v)=-16(4\mathbf{c}_4^3+27\mathbf{c}_6^2)$. When $v=0$, we obtain {an elliptic curve isomorphic to $F$}. The isomorphism $\phi_{u,v}$ can also in principle be made explicit, see \cite{fisher2012hessian}, but we do not go into more detail since we do not need this explicit description. 

\begin{definition}
 Let $X_{\omega,F}(3)(k)$ be the set of isomorphism classes of triples $(E,\omega,\phi)$ consisting of an elliptic curve $E/k$, a regular {non-zero} $1$-form $\omega$, and a symplectic isomorphism $\phi \colon {F[3](\overline{k})\xrightarrow[]{} E[3](\overline{k})}$. 
\end{definition}
Let $\mb{A}^2_{X_F(3)}=\{(u,v) \in \mb{A}^2 : \mathbf{disc}(u,v) \ne 0\}$. The following proposition then follows from \cite[Theorem 13.2]{fisher2012hessian}.
\begin{proposition} \label{twisted X(3)}
	There is a bijection $\phi_F \colon \mb{A}^2_{X_F(3)}(k) \xrightarrow[]{} X_{\omega,F}(3)(k)$ defined by
	\[
	(u,v) \mapsto \left(y^2=x^3+\mathbf{c}_4(u,v)x+\mathbf{c}_6(u,v),\frac{dx}{y},\phi_{u,v}\right).
	\] 
\end{proposition}
\subsection{Covering maps}
A point of the modular curve $X(3)$ is represented by an elliptic curve $E$ with a marked pair of non-trivial $3$-torsion points $P,Q \in E[3](\overline{k})$. By forgetting the point $Q$, we can view the (isomorphism class of) pair $(E,P)$ as a point on the modular curve $X_1(3)$, and by forgetting the point $P$, we view (the isomorphism class of) curve $E$ as a point on the modular curve $X(1)$. These maps can also be defined on the moduli spaces of Weierstrass equations defined above. We identify $X_{\omega}(3)= \mb{A}^2_{X(3)}$,$X_{\omega,1}(3)= \mb{A}^2_{X_1(3)}$ and $X_{\omega}= \mb{A}^2_{X}$, thanks to Propositions \ref{lem:Xomega}, \ref{marked 3-tors point parametrization prop}, \ref{split X(3)}.

\begin{lemma} \label{Jacobian lemma}
 The map $\pi_1 \colon \mb{A}^2_{X_1(3)} \xrightarrow[]{} \mb{A}^2_{X}$ corresponding to $(E,\omega,P) \mapsto (E,\omega)$ is given by:

 \[
 \pi_1(u,v)=\left(8u v-\frac{u^4}{3}, 16v^2-\frac{8vu^3}{3}+\frac{2u^6}{27}\right).
 \] 
 The Jacobian determinant of $\pi_1(u,v)$ is $J(u,v)=256 v^2$. 
 The map $\pi_2 \colon\mb{A}^2_{X(3)}\xrightarrow[]{} \mb{A}^2_{X}$ corresponding to $(E,\omega,\phi) \mapsto (E,\omega)$ is given by 
	\[
	\pi_2(u,v) =\left(-216u^3v - 27v^4,-432u^6 -
	1080u^3v^3 + 54v^6\right).
	\]
	The Jacobian determinant of this map at a point $(u,v)$ is given by
	\[
	J(u,v)=-559872 u^2\left(u-v\right)^2\left(u^2+uv+v^2\right)^2.
	\]
\end{lemma}
\begin{proof}
 A straightforward calculation, which we carry out using the computer algebra system MAGMA \cite{magma}. For the map $\pi_1$, we use formulas from \cite[Section III.1]{silverman2009arithmetic} to convert an elliptic curve in Tate normal form (the one defined in Proposition \ref{marked 3-tors point parametrization prop}) to short Weierstrass form, which gives the equations defining $\pi_1$, from which we can directly compute the Jacobian determinant. We use the formula for the Weierstrass form of an elliptic curve in the Hesse normal form given in Corollary \ref{cor:bij3structure} to compute the map $\pi_2$ and its Jacobian determinant. 
\end{proof}
\begin{lemma} \label{preimage lemma}
 Let $(E,\omega) \in X_{\omega}(k)$ be defined over a perfect field $k$. There is a bijection between the set of $k$-rational points in the preimage $\pi_1^{-1}(E,\omega) \cap {X_{\omega,1}(3)(k)}$ and the set $E[3](k) \setminus \{0_E\}$ of non-zero $k$-rational $3$-torsion points of $E$, and a bijection between the set of $k$-rational points in the preimage $\pi_2^{-1}(E,\omega) \cap {X_{\omega}(3)(k)}$ and the set of $k$-rational isomorphisms $\Z/3\Z \times \mu_3 \xrightarrow{} E[3](\overline{k})$.
\end{lemma}
\begin{proof}
 By definition, we have $\pi_1^{-1}(E,\omega)=\{(E,\omega,P) : P \in E[3](\overline{k}), P \ne 0_E\}$. If $P_1,P_2$ are distinct non-zero $3$-torsion points, then the triples $(E,\omega,P_1)$ and $(E,\omega,P_2)$ are not isomorphic, since the only automorphism of $E$ that fixes $E$ and $\omega$ is the identity. The same argument applies to the map $\pi_2$. As this bijection is compatible with the action of $\mathrm{Gal}(\overline{k}/k)$, the rationality statement follows. 
\end{proof}
These covering maps can also be defined for the twisted moduli spaces $X_{\omega,\psi}(3)$ and $X_{\omega,F}(3)$. We make explicit the case when $F\colon y^2=x^3+b$ has $j$-invariant zero. 

\begin{lemma}\label{lem:explchar}
 Let $\psi \colon \mathrm{Gal}(\overline{k}/k) {\xrightarrow{} \{\pm1\}}$ be the quadratic character associated with $k(\sqrt{\lambda})/k$. The map 
 \[
 \pi_\psi \colon \mb{A}^2_{X_{\psi}(3)} \xrightarrow[]{} \mb{A}^2_{X}
 \]
 corresponding to ${\pi_\psi}(E,\omega,P)=(E,\omega)$ is given by:

 \[
 \pi_{\psi}(u,v)=\left(2\lambda u v-\frac{\lambda^2u^4}{3},\lambda v^2-\frac{2\lambda^2vu^3}{3}+\frac{2\lambda ^3u^6}{27}\right).
 \] 
 The Jacobian determinant of $\pi_{\psi}(u,v)$ is $J(u,v)=4 \lambda^2 v^2$. 
 
\end{lemma}
\begin{proof}
 Exactly the same as the first part of Lemma \ref{Jacobian lemma}, using Proposition \ref{3-isogeny parametrization}.
\end{proof}

 \begin{lemma} \label{twisted X(3) jacobian lemma}
 Let $F \colon y^2=x^3+b$ be an elliptic curve. The map $\pi_F \colon\mb{A}^2_{X_F(3)} \xrightarrow[]{} \mb{A}^2_{X}$ defined by $(E,\omega,\phi) \mapsto (E,\omega)$ is given by 
	\[
	\pi_F(u,v) =\left(-2^8\cdot 3^5\cdot b^2v^4 + 2^3 \cdot 3^2 \cdot b u^3v, - 2^{13}\cdot 3^6 \cdot b^3v^6 - 2^7\cdot3^3\cdot5 \cdot b^2u^3v^3 + bu^6\right).
	\]		
 \end{lemma}
 \begin{proof}
 Similar to Lemma \ref{Jacobian lemma}. Note that the formula for the map $\pi_F$ is obtained by specializing the formulas of Proposition \ref{twisted X(3)} to the case $a=0$.
 \end{proof}
 Finally, we also note that an analog of Lemma \ref{preimage lemma} is also true for these twisted parametrizations, and can be proven in the same way.

	\section{Proof of Theorem~\ref{thm:main}}\label{sec:3tors}

 We assume $p \geq 5$ in this section. The case $p=3$ is treated in Section \ref{p-torsion section}.

\begin{lemma}\label{lem:l-divides-iff-l-tors}
Let $E/\Q_p$ be an elliptic curve in minimal form and $\ell \ne p$ a prime. Consider the associated filtration $E(\Q_p) \supset E_0(\Q_p) \supset E_1(\Q_p)$, where $E_0(\Q_p)$ is the subgroup of points with non-singular reduction modulo $p$ and $E_1(\Q_p)$ is the kernel of the reduction modulo $p$. There exists a non-zero $P \in E[\ell](\Q_p)$ if and only if $\ell\mid [E(\Q_p):E_1(\Q_p)]$.
\end{lemma}
\begin{proof}
 Since $p \geq 3$, we have $E_1(\Q_p) \cong \Z_p$ as a topological group by the theory of formal groups \cite[Proposition VII.2.2]{silverman2009arithmetic}. Hence multiplication by $\ell$ induces an isomorphism $E_1(\Q_p) \xrightarrow{\cdot \ell} E_1(\Q_p)$. Thus, $E(\Q_p) \xrightarrow{\cdot \ell} E(\Q_p)$ is an injection if and only if $E(\Q_p)/E_1(\Q_p) \xrightarrow{\cdot \ell} E(\Q_p)/E_1(\Q_p)$ is an injection, and the claim follows. 
\end{proof}

{Let $W_{\gr}$, $W_{\smr}$, $W_{\nsmr}$, and $W_{\arr}$ be subsets of ${\mb{A}^2_X(\Z_p)}$ consisting of those $[a,b]$ for which $E\colon y^2=x^3+ax+b$ is in minimal Weierstrass form and has good reduction, split multiplicative reduction, non-split multiplicative reduction, and additive reduction, respectively, and has a non-trivial $\Q_p$-rational $3$-torsion point. We now compute $\mu_{\Z_p^2}(W_{\gr})$, $\mu_{\Z_p^2}(W_{\smr})$, $\mu_{\Z_p^2}(W_{\nsmr})$, and $\mu_{\Z_p^2}(W_{\arr})$. In Section~\ref{sec:iso}, Proposition~\ref{prop:assume-short-WE}, which is independent of Section~\ref{sec:3tors} we show that these quantities are equal to their analogues for general Weierstrass equations $y^2+a_1xy+{a_3y}=x^3+a_2x^2+a_4x+a_6$.}

\subsection{Good reduction}\label{sub:good-reduction}

\begin{lemma} \label{good reduction count}
 Let $W_{3,p}$ be the set of short Weierstrass equations over $\F_p$ with {an $\F_p$-rational} non-trivial $3$-torsion point. If $p \equiv 1 \pmod{3}$, then $|W_{3,p}|=\frac{3p^2-4p+1}{8}$. If $p \equiv 2 \pmod{3}$, then $|W_{3,p}|=\frac{(p-1)^2}{2}$. 
 \end{lemma}
 \begin{proof}
 Suppose first that $p\equiv 2 \pmod{3}$. Then we cannot have ${E[3](\F_p)} \cong (\Z/3\Z)^2$, since $\zeta_3\notin \F_p$. Thus if $3\mid \#(E(\F_p))$, 
 we have {$E[3](\F_p)\cong\Z/3\Z$}. By Lemma \ref{preimage lemma} $\pi_1 \colon X_{\omega,1}(3)(\F_p) \xrightarrow{} X_{\omega}(\F_p)$ is a 2-to-1 surjective map. Identifying $X_{\omega,1}(3)=\mb{A}^2_{X_1(3)}$, that is defined in Equation \eqref{eq:A21}, we see that $\#(X_{\omega,1}(3)(\F_p))=p^2-p-p+1=(p-1)^2$, we conclude ${|W_{3,p}|}=(p-1)^2/2$. 
 
 Suppose that $p \equiv 1 \pmod{3}$ now. Then the field {$\F_p$} contains every cube root of unity, and so $\mu_3 \times \Z/3\Z \cong (\Z/3\Z)^2$ as a {$\mathrm{Gal}(\overline{\F}_p/\F_p)$}-module. We compute the number $W_{(\Z/3\Z)^2,p}$ of short Weierstrass equations with $E[3](\F_p) \cong (\Z/3\Z)^2$ by counting $\F_p$-points on the moduli space $X_{\omega}(3)$. By Lemma \ref{preimage lemma}, $\pi_2 \colon X_{\omega}(3)(\F_p) \xrightarrow[]{} X_{\omega}(\F_p)$ is a 24-to-1 cover, since if $E[3](\F_p) \cong (\Z/3\Z)^2$, then symplectic isomorphisms $\phi \colon M \xrightarrow[]{} E[3](\overline{\F_p})$ can be identified with $\mathrm{SL}_2(\F_3)$, which has 24 elements. 
 
 Identifying $X_{\omega}(3)(\F_p)=\mb{A}^2_{X(3)}(\F_p)$, that is defined in Equation \eqref{eq:A2}, we see that $\#(X_{\omega}(3)(\F_p))=p^2-4p+3$ and hence $W_{(\Z/3\Z)^2,p}=(p^2-4p+3)/24$. We next consider $\pi_1 \colon X_{\omega,1}(3)(\F_p) \xrightarrow{} X_{\omega}(\F_p)$. Weierstrass equations with $E[3](\F_p)=\Z/3\Z$ have preimages of size 2 under $\pi_1$, while those with $E[3](\F_p)=(\Z/3\Z)^2$ have preimages of size 8. Since we know there are $(p^2-4p+3)/24$ of the latter ones, we know that their preimage in $X_{\omega,1}(3)(\F_p)$ consists of $(p^2-4p+3)/3$ points. This leaves $(p-1)^2-(p^2-4p+3)/3=\frac{2p^2-2p}{3}$ points which map to Weierstrass equations with $E[3](\F_p) \cong (\Z/3\Z)$. We conclude that there are $\frac{p^2-p}{3}$ such equations, and in total 
 \[
 {|W_{3,p}|}=\frac{p^2-p}{3}+\frac{p^2-4p+3}{24}=\frac{3p^2-4p+1}{8}.
 \]
 \end{proof}

{By Lemma~\ref{lem:l-divides-iff-l-tors} a pair $(A,B) \in \Z_p^2$ is an element of $W_{\gr}$ if and only if its reduction $(\tilde{A},\tilde{B})$ mod $p$ is an element of $W_{3,p}$ because short non-minimal Weierstrass equations reduce to the singular curve $y^2=x^3/\F_p$. The following proposition then follows by Lemma ~\ref{good reduction count}}.
	
	\begin{proposition}\label{thm:gr}
		We have 
		\[
		{\mu_{\Z_p^2}}(W_{\gr})={\frac{|W_{3,p}|}{p^2}=\begin{cases} 
			\frac{3p^2-4p+1}{8p^2} &\text{ if } p\equiv 1 \pmod 3, \\
			\frac{(p-1)^2}{2p^2} &\text{ if } p\equiv 2 \pmod 3. \\
		\end{cases}}
		\]
	\end{proposition}

\subsection{Bad reduction}\label{sub:bad-reduction}
Define $\Tilde{E}_{ns}(\mb{F}_p)$ as the reduction modulo $p$ of the points in $E(\Q_p)$ that reduce to non-singular points modulo $p$ and recall $\Tilde{E}_{ns}(\mb{F}_p)\cong E_0(\Q_p)/E_1(\Q_p)$.
\begin{proposition}\label{thm:smr}
 We have 
		\[
		{\mu_{\Z_p^2}}(W_{\smr})={\begin{cases}
			\frac{p-1}{2p^2} &\text{ if }p\equiv 1\pmod 3,\\
			\frac{p-1}{2p^2(p^2+p+1)} &\text{ if }p\equiv 2\pmod 3.
		\end{cases}}
		\]
	\end{proposition}
	
\begin{proof}		
Since $E$ has split multiplicative reduction, $\Tilde{E}_{ns}(\mb{F}_p) \cong \mb{Z}/(p-1)\mb{Z}$ is cyclic of order $p-1$. Thus, if $ p \equiv 1 \pmod{3}$, by Lemma~\ref{lem:l-divides-iff-l-tors}, there is a non-trivial $3$-torsion point in $E_0(\Q_p)$.

If $p \equiv 2 \pmod{3}$, by Lemma~\ref{lem:l-divides-iff-l-tors}, there is a non-trivial $3$-torsion point in $E(\Q_p)$ if and only if {$3\mid n$}, where $n$ is such that $E(\Q_p)/ E_0(\Q_p) \cong \mb{Z}/n\mathbb{Z}$, i.e., $n$ is equal to the valuation of the discriminant of $[a_1,a_2,a_3,a_4,a_6]$ (or equivalently equal to the minus of the valuation of the $j$-invariant), see \cite[Corollary C.15.2.1]{silverman2009arithmetic}. In this case, $E$ has Kodaira symbol $I_n$. Thus, ${\mu_{\Z_p^2}}(W_{\smr})$ is the probability that {$3\mid n$}, or equivalently the probability that the Kodaira symbol of $[a_1,a_2,a_3,a_4,a_6]$ is $I_{3k}$ for some $k \geq 1$, and that the reduction of $E$ is split.
		
The problem is thus reduced to computing the probability that $[a_1,a_2,a_3,a_4,a_6]$ has Kodaira symbol $I_{m}$ and split reduction, for arbitrary {$m \geq 1$}. This quantity is
computed in \cite[Theorem 5.3]{cremona2020localpub} to be equal to $\frac{1}{2}(p-1)^2/p^{m+2}$; see also \cite[p. 451]{cremona2020localpub} for this formula. We obtain the result by summing these values over all $m\geq 1$ for $p\equiv 1\pmod{3}$ and $3m$, for $m\geq 1$ for $p\equiv 2\pmod{3}$.
\end{proof}
\begin{proposition}\label{thm:nsmr}
		We have 
		\[
		{\mu_{\Z_p^2}}(W_{\nsmr})={\begin{cases}
			0 &\text{ if }p\equiv 1\pmod 3,\\
			\frac{p-1}{2p^2} &\text{ if }p\equiv 2\pmod 3.
		\end{cases}}
		\]
	\end{proposition}
	\begin{proof}
In this case, $\#\Tilde{E}_{ns}(\mb{F}_p)=p+1$ (see \cite[Exercise 3.5.a.ii]{silverman2009arithmetic}), so for $p\equiv 2\pmod{3}$, by Lemma~\ref{lem:l-divides-iff-l-tors} there is a non-trivial $3$-torsion point in $E(\Q_p)$. Then ${\mu_{\Z_p^2}}(W_{\nsmr})$ is equal to the probability that $[a_1,a_2,a_3,a_4,a_6]$ defines a curve with non-split multiplicative reduction, which is given in \cite[Theorem 5.3]{cremona2020localpub}.

If $p\equiv 1\pmod{3}$, since $\#(E(\Q_p)/ E_0(\Q_p))$ divides $4$ (see \cite[Corollary C.15.2.1]{silverman2009arithmetic}), by Lemma~\ref{lem:l-divides-iff-l-tors} there is {no non-trivial} $3$-torsion point in $E(\Q_p)$.
\end{proof}

 \begin{proposition}\label{thm:ar}
 We have 
		\[
		{\mu_{\Z_p^2}}(W_{\arr})=\frac{p-1}{2p^5}+\frac{p-1}{2p^{8}}.
		\] 
	\end{proposition}
	\begin{proof}
		Since $\#\Tilde{E}_{ns}(\mb{F}_p) \mid p$, by Lemma~\ref{lem:l-divides-iff-l-tors} there is a non-trivial $3$-torsion point in $E(\Q_p)$ if and only if $3\mid [E(\Q_p): E_0(\Q_p)]$. This happens if and only if the Tamagawa number $c_p$ of $E$ is $3$, which happens only if the Kodaira symbol is $IV$ or $IV^*$, and these probabilities are given in \cite[Proposition 2.2 and the proof of Theorem 5.3]{cremona2020localpub}. See in particular \cite[Page 453 and 456]{cremona2020localpub}.
	\end{proof}

{\begin{proof}[Proof of Theorem~\ref{thm:main}] We distinguish cases of good, split and non-split multiplicative, and additive reduction, and compute the probability in each separate case. These probabilities are conditional on being in minimal form, which, by \cite[Proposition 2.1]{cremona2020localpub} happens with probability $1-p^{-10}$. To remove the conditional probability, we need to multiply our results by $\frac{1}{1-p^{-10}}$, which is justified by \cite[Proposition 2.6]{cremona2020localpub}. The theorem follows by adding the results from Propositions \ref{thm:gr}, \ref{thm:smr}, \ref{thm:nsmr}, and \ref{thm:ar}.
\end{proof}}

\clearpage
 
	\section{Isogenies of degree \texorpdfstring{$3$}{}}\label{sec:iso}
 \subsection{\texorpdfstring{$p$}{}-adic integration}
 Fix $p>3$. Let $\abs{\cdot }_{p}$ denote the $p$-adic absolute value on $\Q_{p}$, normalized so that $\abs{p}_p=\frac{1}{p}$. Let $r \in \mathbb{N}$ and let $dx$ denote the Haar measure on $\Z_p^r$, normalized {so that} $\mu_{\Z_p^r}(\Z_p^r)=1$.
Computing the $p$-adic density of a measurable subset $S \subset \Z_p^r$ amounts to computing the integral $\int_{\Z_p^r} \mathbf{1}_{S} \,dx$ of its indicator function $\mathbf{1}_{S}$. The following theorem of Igusa provides a change of variables formula for these integrals. We will only use the result for polynomials.

\begin{theorem}[Igusa] \label{change of vars}
 Let $f\colon \Z_p^r \xrightarrow[]{} \Z_p^r$ {be} given by $f(x)=(f_1(x),\ldots,f_r(x))$, where $f_i(x)$ are power series in $\Z_p[[x_1,\ldots,x_r]]$ that converge uniformly on $\Z_p^r$. Let $J(x)\coloneq\partial(f_1,\ldots,f_r)/\partial(x_1,\ldots,x_r)$ be the Jacobian determinant of $f$. Suppose that for $a \in \Z_p^r$ we have {$\partial(f_1,\ldots,f_r)/\partial(x_1,\ldots,x_r)|_{a} \ne 0$.} Then there are open neighbourhoods $U$ of $a$ and $V$ of $y=f(a)$ such that the restriction $f \colon U \xrightarrow[]{} V$ is a bijection and
 \[
 dy={\abs{J(x)}_p} \,dx,
 \]
 where $dy$ is the restriction of the Haar measure to $V$ and {$dx$} the restriction to $U$. For any integrable $\mathbb{R}$-valued function $g$ on $V$, we have
 \[
 \int_{V} g(y) \,dy=\int_{U} g(f(x)) \abs{J(x)}_{p} \,dx.
 \]
\end{theorem}
\begin{proof}
 See Proposition 7.4.1 of \cite{igusa2000introduction} for a proof and Chapter 2 of \cite{igusa2000introduction} for the definition of convergent power series and analytic functions in this context. 
\end{proof}

The importance of $p$-adic integrals for the problem can already be seen from the following proposition which allows us to look at different types of Weierstrass equations.

\begin{proposition}\label{prop:assume-short-WE}
Let $p>3$ be a prime, and $P$ a property of elliptic curves over $\Q_p$ that does not depend on the choice of a Weierstrass equation. 
Let 
\[
S_1=\left\{[a_1,a_2,a_3,a_4,a_6] \in \Z_p^5 \;\colon\; E\colon\, y^2+a_1xy+a_3y=x^3+a_2x^2+a_4x+a_6\;\text{has property $P$}\right\};
\]
\[
S_2=\left\{[a_2,a_4,a_6] \in \Z_p^3 \;\colon\; E\colon\, y^2=x^3+a_2x^2+a_4x+a_6\;\text{has property $P$}\right\};
\]
\[
S_3=\left\{[a_4,a_6] \in \Z_p^2\; \colon\; E\colon\, y^2=x^3+a_4x+a_6\;\text{has property $P$}\right\}.
\]
{Assume that the $S_i$ are measurable}. 
Then $\mu_{\Z_p^5}(S_1)=\mu_{\Z_p^3}(S_2)=\mu_{\Z_p^2}(S_3)$. If $p=3$, then $\mu_{\Z_p^5}(S_1)=\mu_{\Z_p^3}(S_2)$.
\end{proposition}
\begin{proof}
The change of variables that transforms a shape from $S_3$ to $S_2$ is given by $\psi_{a_2}\colon x\mapsto x-\frac{a_2}{3}$, which maps $(a_4,a_6)$ to $(a'_4,a'_6)$ and it is a bijection whose Jacobian determinant is 1, for any $a_2\in\Z_p$. Hence, 
\[
\mu_{\Z_p^3}(S_2)=\int_{\Z_p^3} 1_{S_2}d\mu_{\Z_p^3}=\int_{a_2\in \Z_p}\int_{(a'_4,a'_6)\;\colon\; (a_2,a'_4,a'_6)\in S_2}d\mu_{\Z_p^2}d\mu_{\Z_p}=\int_{a_2\in\Z_p}\int_{(a_4,a_6)\in S_3}d\mu_{\Z_p^2}d\mu_{\Z_p}=\mu_{\Z_p^2}(S_3).
\]
To finish the proof we show $\mu_{\Z_p^5}(S_1)=\mu_{\Z_p^3}(S_2)$ for $p\geq 3$. We use a change of variables $y\mapsto y-\frac{a_1x}{2}-\frac{a_3}{2}$, which maps $(a_2,a_4,a_6)$ to $(a'_2,a'_4,a'_6)$, it is a bijection, and the determinant of the Jacobian of this transformation is $1$, for any $a_1,a_3\in \Z_p$. Then, in the same way, we prove that $\mu_{\Z_p^5}(S_1)=\mu_{\Z_p^3}(S_2)$.
\end{proof}

So, from now on, we see that we may consider only short Weierstrass equations to study the $3$-torsion (or $\ell$-torsion) points and similar questions.

\begin{remark}
We will mainly consider integrals of the form $\int_{\Z_p^r} \abs{f(x)}^s_{p} \,dx$, where $f(x) \in \Z[x]$ is a polynomial. Integrals of this type are known as Igusa zeta functions and there is extensive literature on their properties; see \cite{denef1990report} and references cited therein. 
\end{remark}

\begin{lemma} \label{p-adic integral lemma more general}
{Let $k, m, n\in \Z_{\geq 0}$. Consider $f(x,y)=g(x,y)^kh(x,y)^m$, where $h(x,y)=\prod^{n+1}_{i=1} l_i(x,y) $, where $l_i(x,y)$ are linear forms in $\Z_p[x,y]$ with pairwise distinct roots modulo $p$, and $g\in \Z_p[x,y]$ is homogeneous of degree $d>1$ such that the reduction $\Tilde{g}(x,y)\neq 0$ is irreducible in $\F_p[x,y]$.} Then
\[
\int_{x,y \in \Z_p^2} \abs{f(x,y)}_{p} \,dx\,dy = \frac{p^{(n+1)m+kd}(p^{m+2}-p^{m+1}n+pn-1)}{(p^{(n+1)m+kd+2}-1)(1+p+\ldots+p^{m})}.
\]
\end{lemma}
\begin{proof}
Since $\Tilde{g}(x,y)$ is irreducible of degree $d>1$, we have $p\mid g(x,y)$ if and only if $(x,y) \in (p\Z_p)^2$. Denote $I=\int_{x,y \in \Z_p^2} \abs{f(x,y)}_p \,dx\,dy$. Then, 
\[
I=\int_{(x,y) \in (p\Z_p)^2} \abs{f(x,y)}_{p} \,dx\,dy +\int_{(x,y) \in \Z_p^2 \setminus (p\Z_p)^2} \abs{h(x,y)}^m_{p}\,dx\,dy.
\]
The change of variables $(x,y) \mapsto (px,py)$ shows that 
\[
\int_{(x,y) \in (p\Z_p)^2} \abs{f(x,y)}_{p} \,dx\,dy=p^{-(n+1)m-kd-2}I.
\]
Let $(x_i:y_i)$ be a zero of $l_i(x,y)$, with $x_i$ and $y_i$ scaled so that $(x_i,y_i) \in \Z_p^2$ and $(x_i,y_i) \not\in (p\Z_p)^2$, for $1 \leq i \leq n+1$. Then 
\begin{align*}
&\int_{(x,y)\in \Z_p^2\backslash (p\Z_p)^2}\abs{h(x,y)}_p^m\,dx\,dy\\= &\left(\sum^{n+1}_{i=1} \int_{\substack{(x,y)\in \Z_p^2\backslash (p\Z_p)^2 \\ (x:y) \equiv (x_i:y_i) \pmod{p}}} \abs{l_i(x,y)}_p^m \,dx\,dy \right)+ \left( \int_{\substack{(x,y)\in \Z_p^2\backslash (p\Z_p)^2 \\ (x:y) \not\equiv (x_i:y_i) \pmod{p} \\ \text{for all $i$}}} \,dx\,dy \right).
\end{align*}
Here we used the assumption that $(x_i:y_i)$ are pairwise distinct modulo $p$. The second bracket is readily computed to be $ {(p-n)(p-1)/p^2}$. We compute the sum in the first bracket term by term. Fix an $i \leq n+1$. We want to compute the integral
\[
\int_{(x:y) \equiv (x_i:y_i) \pmod{p}} \abs{l_i(x,y)}_p^m \,dx\,dy.
\]
By the change of variables theorem, this integral is invariant under $\mathrm{SL}_2(\Z_p)$-substitutions of $x$ and $y$. We are free to assume that $l_i(x,y)=x$ and the above integral reduces to

\begin{align*}
\int_{\substack{(x,y)\in \Z_p^2 \\ (x:y) \equiv (0:1) \pmod{p}}} \abs{x}_p^m \,dx\,dy&= \int_{x \in p\Z_p, y \in \Z_p\backslash p\Z_p} \abs{x}_p^m \,dx\,dy\\
&=\left(1-\frac{1}{p}\right) \int_{x \in p\Z_p} \abs{x}_p^m dx=\left(\frac{p-1}{p}\right)^2\frac{1}{p^{m+1}-1}.
\end{align*}
The lemma now follows by a direct calculation.

\end{proof}

\subsection{Probability for \texorpdfstring{$\Q_p$}{}-rational \texorpdfstring{$3$}{}-torsion using \texorpdfstring{$p$}{}-adic integration} \label{integration proof of 3-torsion}
As a warm-up for computing {the probability of having a $\Q_p$-rational $3$-isogeny} and as a sanity check on our previous calculations, we give a second proof of Theorem \ref{thm:main} using $p$-adic integration. We introduce some notation. As in Section \ref{moduli space section}, let $\mb{A}^2_{X}(\Q_p) \subset \Q_p^2$ be the space of pairs $(a,b)$ such that $E_{a,b} \colon y^2=x^3+ax+b$ is an elliptic curve. We compute the Haar measure of the set 
\[
W_{\geq \Z/3\Z}=\left\{ (a,b) \in \mb{A}^2_{X}(\Z_p) : \#\left(E_{a,b}[3](\Q_p)\right)\geq 3\right\}.
\] 
Let $W^{\Q_p}_{\geq \Z/3\Z}$ be the set of pairs $(a,b) \in \mb{A}^2_{X}(\Q_p)$ such that $\#(E_{a,b}[3](\Q_p))\geq 3$, let $W^{\Q_p}_{\Z/3\Z \times \mu_3}$ be the set of pairs $(a,b) \in \mb{A}^2_{X}(\Q_p)$ such that {$E_{a,b}[3]({\overline{\Q_p}}) \cong \Z/3\Z \times \mu_3 $} as a $\mathrm{Gal}(\overline{\Q_p}/\Q_p)$-module, and let $W^{\Q_p}_{\Z/3\Z}$ be the set of pairs $(a,b) \in \mb{A}^2_{X}(\Q_p)$ such that $E_{a,b}[3](\Q_p) \cong \Z/3\Z$. Define $W_{\Z/3\Z \times \mu_3}=W^{\Q_p}_{\Z/3\Z \times \mu_3} \cap \Z_p^2$ and $W_{ \Z/3\Z}=W^{\Q_p}_{\Z/3\Z} \cap \Z_p^2$ to be the subsets of integral pairs $(a,b)$. 

There are two main cases, depending on $p$ modulo $3$. When $p \equiv 1 \pmod 3$, then $\Z/3\Z \times \mu_3 \cong (\Z/3\Z)^2$ as $\mathrm{Gal}(\overline{\Q_p}/\Q_p)$-module, since $\Q_p$ contains each cube root of unity. Then $W_{\geq \Z/3\Z}=W_{\Z/3\Z}\sqcup W_{\Z/3\Z \times \mu_3} $.

If $p \equiv 2 \pmod{3}$, then {$E_{a,b}[3](\Q_p) \cong (\Z/3\Z)^2$} is impossible for any $(a,b) \in \mb{A}^2_{X}(\Q_p)$, so in this case we have $W_{\geq \Z/3\Z}=W_{ \Z/3\Z}$.

By Proposition \ref{marked 3-tors point parametrization prop}, we have $W^{\Q_p}_{\geq \Z/3\Z}=\pi_1(\mb{A}^2_{X_1(3)}(\Q_p))$, where 
\[
\pi_1 \colon \mb{A}^2_{X_1(3)} (\Q_p) \xrightarrow[]{} \mb{A}^2_{X}(\Q_p)
\]
is the map of Lemma \ref{Jacobian lemma} that forgets the $3$-torsion point: $(E,\omega,P) \mapsto (E,\omega)$. The number of preimages a pair $(E,\omega)$ has under this map depends on whether $E[3](\Q_p)\cong \Z/3\Z$ or $E[3](\Q_p)\cong (\Z/3\Z)^2$, and so we define 
\[
\mb{A}^2_{X_1(3),\Z/3\Z}(\Q_p)=\left\{(a,b) \in \mb{A}^2_{X_1(3)}{(\Q_p)} : {E_{\pi_1(a,b)}}[3](\Q_p)=\Z/3\Z \right\}
\] 
and 
\[
\mb{A}^2_{X_1(3), (\Z/3\Z)^2}(\Q_p)=\left\{(a,b) \in \mb{A}^2_{X_1(3)}{(\Q_p)}  : {E_{\pi_1(a,b)}}[3](\Q_p)=(\Z/3\Z)^2\right\}.
\] 

\begin{lemma} \label{integral covering split X1(3)}
We have $W_{(\Z/3\Z)^i}=\pi_1(\mb{A}^2_{X_1(3),(\Z/3\Z)^i}(\Z_p))$, and the restriction $\pi_1 \colon\mb{A}^2_{X_1(3),(\Z/3\Z)^i}(\Z_p) \xrightarrow[]{} W_{(\Z/3\Z)^i}$ is a $K_i$-to-1 covering map, where $K_1=2$ and $K_2=8$.
\end{lemma}
\begin{proof}
By definition we have $W^{\Q_p}_{\Z/3\Z}=\pi_1(\mb{A}^2_{X_1(3),\Z/3\Z}(\Q_p))$ and $W^{\Q_p}_{(\Z/3\Z)^2}=\pi_1(\mb{A}^2_{X_1(3),(\Z/3\Z)^2}(\Q_p))$.
By Lemma \ref{preimage lemma}, $\pi_1 \colon\mb{A}^2_{X_1(3),(\Z/3\Z)^i}(\Q_p) \xrightarrow[]{} W^{\Q_p}_{(\Z/3\Z)^i}$ is a $K_{i}$-to-1 covering. Consider a Weierstrass equation $y^2=x^3+Ax+B$ over $\Z_p$ with a marked non-zero $3$-torsion point $P\coloneq(x(P),y(P)) \in E[3](\Q_p)$. By the Nagell-Lutz Theorem (see \cite[Theorem VIII.7.1]{silverman2009arithmetic}), we have $x(P),y(P) \in \Z_p$. Hence the procedure of Lemma \ref{tate normal form} that puts $y^2=x^3+Ax+B$ into Tate normal form $y^2+uxy+vy={x^3}$ does not introduce any denominators, so we have $u,v \in \Z_p$. Thus all preimages of $y^2=x^3+Ax+B$ are defined over $\Z_p$.
\end{proof}
\begin{lemma} \label{integral covering split X(3)}
We have $W_{\Z/3\Z \times \mu_3}=\pi_2(\mb{A}^2_{X(3)}(\Z_p))$, and the restriction $\pi_2 \colon \mb{A}^2_{X(3)}(\Z_p) \xrightarrow[]{} W_{\Z/3\Z \times \mu_3}$ is a $K_p$-to-1 covering map, where $K_p=2$ if $p \equiv 2 \pmod{3}$ and $K_p=24$ if $p \equiv 1 \pmod{3}$. 
\end{lemma}

\begin{proof}
If $p \equiv 2 \pmod{3}$, then $\Q_p$ does not contain {a primitive cube root of unity}, and $\Z/3\Z \not\cong \mu_3$ as a $\mathrm{Gal}(\bar{\Q_p}/\Q_p)$-module. The group of $\Q_p$-rational symplectic automorphisms $M=\Z/3\Z \times \mu_3$ is generated by the involution $[-1]$ and isomorphic to $\Z/2\Z$. If $p \equiv 1 \pmod{3}$, then $\Z/3\Z \cong \mu_3$ as $\mathrm{Gal}(\overline{\Q_p}/\Q_p)$-module, and the group of symplectic $\Q_p$-automorphisms of $M$ is $\mathrm{SL}_2(\F_3)$ and has 24 elements. Lemma \ref{preimage lemma} then implies the statement of the lemma with $\Z_p$ replaced by $\Q_p$. To finish the proof it suffices to show that if for some $(u,v) \in \mb{A}^2_{X(3)}(\Q_p)$ we have $\pi_2(u,v) \in \Z_p^2$, then we have $(u,v) \in \Z_p^2$. 

By Lemma \ref{Jacobian lemma} we have 
\[
\pi_2(u,v)=(c_4(u,v),c_6(u,v))=\left(-216u^3v - 27v^4,-432u^6 -
	1080u^3v^3 + 54v^6\right).
 \]
 Suppose that $(u,v) \not\in \Z_p^2$, then by clearing denominators we can find $(u',v')=(p^{k}u,p^{k}v)$ for $k>0$ with $(u',v') \in \Z_p^2$ and
 $(u',v') \not\in (p\Z_p)^2$. Thus,
 \[
 \pi_2(u',v')= \left(p^{4k}\left(-216u^3v - 27v^4\right),p^{6k}\left(-432u^6 -
	1080u^3v^3 + 54v^6\right)\right).
 \]

 The resultants of $c_4(u,v)$ and $c_6(u,v)$, viewed as polynomials in $u$ and $v$ respectively, are $2^{12} 3^{36} v^{24}$ and $-2^{16} 3^{39} u^{24}$. Since we have $p\mid c_4(u',v')$ and $p\mid c_6(u',v')$ we see that $p\mid u'$ and $p\mid v'$, contradicting our assumption that $(u',v') \not\in (p\Z_p)^2$.
\end{proof}

\begin{lemma} \label{overcounting lemma}The following hold.
\begin{enumerate}[label=(\roman*)]
\item We have
\[
\mu_{\Z_p^2}\left(W_{\geq \Z/3\Z}\right)=\frac{1}{2}\int_{\mb{A}^2_{X_1(3)}(\Z_p)} \abs{J(\pi_1)(u,v)}_{p} \,du\,dv-3\mu_{\Z_p^2}\left(W_{(\Z/3\Z)^2}\right).
\]

 \item 
When $p \equiv 1 \pmod{3}$, we have 
\[
\mu_{\Z_p^2}\left(W_{(\Z/3\Z)^2}\right)=\mu_{\Z_p^2}\left(W_{\Z/3\Z\times \mu_3}\right)=\frac{1}{24}\int_{\mb{A}^2_{X(3)}(\Z_p)}\abs{J(\pi_2)(u,v)}_{p} \,du\,dv,
\]
and so 
\[
\mu_{\Z_p^2}(W_{\geq \Z/3\Z})=\frac{1}{2}\int_{\mb{A}^2_{X_1(3)}(\Z_p)} \abs{J(\pi_1)(u,v)}_{p} \,du\,dv-\frac{1}{8}\int_{\mb{A}^2_{X(3)}(\Z_p)}\abs{J(\pi_2)(u,v)}_p \,du\,dv.
\]
\item When $p \equiv 2 \pmod{3}$, we have $\mu_{\Z_p^2}(W_{(\Z/3\Z)^2})=0$ and so
\[
\mu_{\Z_p^2}(W_{\geq \Z/3\Z})=\frac{1}{2}\int_{\mb{A}^2_{X_1(3)}(\Z_p)} \abs{J(\pi_1)(u,v)}_p \,du\,dv.
\]
\end{enumerate}

\end{lemma}
\begin{proof}
Similar to Lemma \ref{good reduction count}. We partition 
\[
\mu_{\Z_p^2}\left(W_{\geq \Z/3\Z}\right)=\mu_{\Z_p^2}\left(W_{ \Z/3\Z}\right) + \mu_{\Z_p^2}\left(W_{(\Z/3\Z)^2}\right).
\]
Let $p\equiv 1\pmod 3$. By Theorem \ref{change of vars} and Lemma \ref{integral covering split X1(3)}, we have 
\[
\mu_{\Z_p^2}\left(W_{ \Z/3\Z}\right)=\frac{1}{2}\int_{\mb{A}^2_{X_1(3),\Z/3\Z}(\Z_p)}\abs{J(\pi_1)(u,v)}_{p} \,du\,dv
\]
and 
\[
\mu_{\Z_p^2}\left(W_{(\Z/3\Z)^2}\right)=\frac{1}{8}\int_{\mb{A}^2_{X_1(3),(\Z/3\Z)^2}(\Z_p)}\abs{J(\pi_1)(u,v)}_{p} \,du\,dv,
\]
so
\begin{align*}
\mu_{\Z_p^2}\left(W_{\geq \Z/3\Z}\right)&=\frac{1}{2}\int_{\mb{A}^2_{X_1(3),\Z/3\Z}(\Z_p)}\abs{J(\pi_1)(u,v)}_{p} \,du\,dv+\frac{1}{8}\int_{\mb{A}^2_{X_1(3),(\Z/3\Z)^2}(\Z_p)}\abs{J(\pi_1)(u,v)}_{p} \,du\,dv\\
&=\frac{1}{2}\int_{\mb{A}^2_{X_1(3)}(\Z_p)}\abs{J(\pi_1)(u,v)}_{p} \,du\,dv-\frac{3}{8}\int_{\mb{A}^2_{X_1(3),(\Z/3\Z)^2}(\Z_p)}\abs{J(\pi_1)(u,v)}_{p} \,du\,dv\\
&=\frac{1}{2}\int_{\mb{A}^2_{X_1(3)}(\Z_p)} \abs{J(\pi_1)(u,v)}_{p} \,du\,dv-3\mu_{\Z_p^2}\left(W_{(\Z/3\Z)^2}\right),
\end{align*}
where the second equality is obtained using 
\[
\mb{A}^2_{X_1(3)}(\Z_p)=\mb{A}^2_{X_1(3),\Z/3\Z}(\Z_p) \sqcup \mb{A}^2_{X_1(3),(\Z/3\Z)^2}(\Z_p).
\]
By Lemma \ref{integral covering split X(3)} and Theorem \ref{change of vars}, we have 
\[
\mu_{\Z_p^2}(W_{\Z/3\Z \times \mu_3})=\frac{1}{24}\int_{\mb{A}^2_{X(3)}(\Z_p)}\abs{J(\pi_2)(u,v)}_{p} \,du\,dv,
\] 
and now (ii) follows. The case $p\equiv 2\pmod 3$ is similar.
\end{proof}

\begin{lemma} \label{integral over X1 lemma}
 We have 
 \[
 \int_{\mb{A}^2_{X_1(3)}(\Z_p)} \abs{J(\pi_1)(u,v)}_{p} \,du\,dv=\frac{p^2}{p^2+p+1}.
 \]
\end{lemma}
\begin{proof}
 By Lemma \ref{Jacobian lemma}, we have $J(\pi_1)(u,v)=256 v^2$. Notice that $ \Z_p^2\setminus \mb{A}^2_{X_1(3)}$ has measure zero. By Lemma \ref{p-adic integral lemma more general}, 
 \[
 \int_{\mb{A}^2_{X_1(3)}(\Z_p)} \abs{J(\pi_1)(u,v)}_{p} \,du\,dv=\int_{\Z_p^2} \abs{v}_p^2 \,du\,dv=\frac{p^2}{p^2+p+1}.
 \]
\end{proof}
\begin{proposition} \label{density of split X(3) equations}
If $p \equiv 1 \pmod{3}$, then
 \[
 \mu_{\Z_p^2}\left(W_{\Z/3\Z \times \mu_3}\right)=\frac{1}{24}\frac{p^{8}(p^{4}-3p^{3}+3p-1)}{(p^{10}-1)(1+p+p^{2})}.
 \]
\end{proposition}

\begin{proof}
 By Lemma \ref{Jacobian lemma} the Jacobian determinant of $\pi_2$ is given by
 \[
	J(u,v)=-2^{8} 3^{7} u^2(u-v)^2(u^2+uv+v^2)^2.
	\]
 By Theorem \ref{change of vars} and Lemma \ref{integral covering split X(3)}, 
 \[ 24\mu_{\Z_p^2}(W_{\Z/3\Z \times \mu_3})=
24\int_{W_{\Z/3\Z \times \mu_3}} d\mu_{\Z_p^2}(a,b)=\int_{ \mb{A}^2_{X(3)}(\Z_p) } \abs{J(u,v)}_{p} d\mu_{\Z_p^2}(u,v).
\]
The complement of $ \mb{A}^2_{X(3)}(\Z_p) \subset \Z_p^2$, the set of pairs $(u,v)$ that define singular Weierstrass equations, has measure zero, and so the above integral is equal to the integral over the entire $\Z_p^2$. Hence,
\[
\int_{\Z_p^2} \abs{J(u,v)}_{p} d\mu_{\Z_p^2}(u,v)=\int_{\Z_p^2} \abs{u^2(u-v)^2(u^2+uv+v^2)^2}_{p} \,du\,dv.
\]
Since $p \equiv 1 \pmod{3}$, we have $u^2+uv+v^2=(u-\zeta_3v)({u-\zeta_3^2v})$ with $\zeta_3\in \Z_p^*$. Now the result follows by Lemma \ref{p-adic integral lemma more general}, with $m=2$, $n=3$, $k=0$, and {$h=u(u-v)(u-\zeta_3v)(u-\zeta_3^2v)$}.

\end{proof}
\begin{theorem}
 When $p \equiv 1 \pmod{3}$, we have 
 \[
 \mu_{\Z_p^2}\left(W_{\geq\Z/3\Z}\right)=\frac{3p^8 + 4p^4 - 4p^3 + 4p^2}{8(p^8 + p^6 + p^4 + p^2 + 1)},
 \]
 and when $p \equiv 2 \pmod{3}$, we have 
 \[
 \mu_{\Z_p^2}\left(W_{\geq\Z/3\Z}\right)=\frac{1}{2}\frac{p^2}{p^2+p+1}.
 \]
\end{theorem}
\begin{proof}
 By Lemma \ref{overcounting lemma}, Lemma \ref{integral over X1 lemma}, and Proposition \ref{density of split X(3) equations},
 \[
 \mu_{\Z_p^2}(W_{\geq\Z/3\Z})=\frac{1}{2}\frac{p^2}{p^2+p+1}-\frac{1}{8}\frac{p^{8}(p^{4}-3p^{3}+3p-1)}{(p^{10}-1)(1+p+p^{2})}=\frac{3p^8 + 4p^4 - 4p^3 + 4p^2}{8(p^8 + p^6 + p^4 + p^2 + 1)}.
 \]
 Case $p\equiv 2 \pmod{3}$ is obtained in the same way.
\end{proof}
\begin{remark}
Notice that this agrees with Theorem \ref{thm:main}.
\end{remark}
\subsection{Probability for \texorpdfstring{$\Q_p$}{}-rational \texorpdfstring{$3$}{}-isogeny}

Define
\[
W_{\iso}(\Q_p)=\{ (a,b) \in \mb{A}^2_{X}(\Q_p) : \text{ {$E_{a,b}$} admits a cyclic $3$-isogeny defined over $\Q_p$} \}
\]
and $W_{\iso}=W_{\iso}(\Q_p)\cap \mb{A}^2(\Z_p)$.
In this section we compute the measure of $W_{\iso}$,
proving Theorem \ref{thm:mainiso}. Recall that we are assuming $p>3$. By Lemma \ref{3-isogeny lemma}, for such an elliptic curve $E_{a,b}$, the kernel of a $\Q_p$-rational $3$-isogeny $\phi\colon E \xrightarrow{} E'$ consists of three points $0_E,P$ and $-P$, with $P$ either defined over $\Q_p$ or over a quadratic extension of $\Q_p$.

In Section \ref{marked torsion point section}, for a quadratic character $\psi$ we have defined moduli spaces $X_{\omega,\psi}(3)$ parametrizing {triples} $(E,\omega,P)$ where the point $P$ is defined over the quadratic field associated with $\psi$, and in Proposition \ref{3-isogeny parametrization} we identified $X_{\omega,\psi}(3)$ with an open subset $\mb{A}^2_{X_{\psi}(3)}$ of the affine plane. Let $\pi_{\psi} \colon \mb{A}^2_{X_{\psi}(3)} \xrightarrow{} \mb{A}^2_{X}$ be the map $(E,\omega,P) \mapsto (E,\omega)$, which we explicitly defined in Lemma \ref{lem:explchar}.

There are three different quadratic extensions $\Q_p(\sqrt{\lambda})$ of $\Q_p$ inside $\overline{\Q_p}$, obtained by taking $\lambda=\nu ,p$ and $\nu p$, where $\nu \in \Z^{*}_{p}$ is a quadratic non-residue modulo $p$. Let $\chi_{\nu },\chi_{p}$ and $\chi_{\nu p}$ be the corresponding quadratic characters of $\mathrm{Gal}(\overline{\Q_p}/\Q_p)$. Together with the case of a $\Q_p$-rational non-trivial $3$-torsion point, they determine the four types of $3$-isogenies over $\Q_p$, and we have 
{
\begin{equation} \label{decomposition count isogeny}
W_{\iso}(\Q_p)=\bigcup_{\psi\in\{\chi_1,\chi_\nu,\chi_p,\chi_{\nu p}\}} \pi_{\psi}\left(\mb{A}^2_{X_{\psi}(3) }(\Q_p)\right)
\end{equation}
where $\pi_{\chi_1}=\pi_1$.
}
Our strategy is to refine this description of $W_{\iso}(\Q_p)$ to a description of the integral subset {$W_{\iso}$}, and compute the measure $\mu_{\Z_p^2}({W_{\iso}})$ similarly as in the $3$-torsion case, by computing integrals of Jacobian functions on a covering space of $W_{\iso}$. 

{With $\mathbf{1}$, we denote the trivial character.}
\begin{lemma} \label{multiple isogeny lemma}
 Suppose $E_{a,b} \colon y^2=x^3+ax+b$ is an elliptic curve over $\Q_p$ with at least two different $3$-isogenies defined over $\Q_p$. Then: 
 \begin{enumerate}[label=(\roman*)]
 \item If $p \equiv 1 \pmod {3}$, then {$E_{a,b}[3](\overline{\Q_p})=\chi \oplus \chi$, where $\chi$ is one of the four quadratic characters $\mathbf{1},\chi_\nu ,\chi_p,\chi_{\nu p}$ of $\mathrm{Gal}(\overline{\Q_p}/\Q_p)$} and $E_{a,b}$ admits four different $3$-isogenies, all with the same character $\chi$. With $E_{a,b}[3](\overline{\Q_p})=\chi \oplus \chi$ we mean that $\Gal(\overline{\Q_p}/\Q_p)$ acts on the {$2$-dimensional} $\F_3$-vector space $E_{a,b}[3](\overline{\Q_p})$ as $\chi \oplus \chi$.
 \item If $p \equiv 2 \pmod{3}$, then $E_{a,b}[3](\overline{\Q_p})=1\oplus \chi_{\nu }$, or $E_{a,b}[3](\overline{\Q_p})=\chi_{p }\oplus \chi_{\nu p}$, and $E_{a,b}$ admits exactly two different $3$-isogenies defined over $\Q_p$.
 \end{enumerate}
\end{lemma}
\begin{proof}
 By assumption, we have $E_{a,b}[3](\overline{\Q_p}) \cong \chi_1 \oplus \chi_2$, for two quadratic characters $\chi_1,\chi_2$ of $\mathrm{Gal}(\bar{\Q_p}/\Q_p)$. By the Galois equivariance of the Weil pairing, we have $\chi_1 \chi_2=\chi_{3,cyc}$, where $\chi_{3,cyc}$ is the modulo $3$ cyclotomic character (see \cite[Proposition III.8.3]{silverman2009arithmetic}). The result follows by observing that $\chi_{3,cyc}$ is trivial if $p \equiv 1 \pmod{3}$ and $\chi_{3,cyc}=\chi_{\nu }$ is the unique unramified character if $p \equiv{2} \pmod{3}$.
\end{proof}

Define
	\[
	\eta_i=\mu_{\Z_p^2}\left(\{[a,b]\in \Z_p^2: y^2=x^3+ax+b \text{ admits exactly $i$ $\Q_p$-rational $3$-isogenies}\}\right)
	\]
 for $i=1,2,4$. We want to compute $\eta_1+\eta_2+\eta_4$. We begin by computing $\eta_2$ and $\eta_4$. 
	
 Consider the elliptic curve $E_{b}\colon y^2=x^3+b$ with $b=16p^{3\delta}\nu ^{3\ve}$ and $\delta,\ve\in\{0,1\}$, where $\nu \in \Z_p^{*}$ is a non-square as before. Let $\chi$ be the quadratic character associated with $\sqrt{\nu ^\ve p^\delta}$. The non-trivial $3$-torsion points of $E_b$ are of the form $(0,\pm 4p^\delta\nu ^\ve\sqrt{\nu ^\ve p^\delta})$ and $(-4\zeta_3^ip^\delta \nu ^\ve,\pm 4p^\delta \nu ^\ve\sqrt{\nu ^\ve p^\delta}\sqrt{-3})$ with $i=0,1,2$. From this we see that we have the decomposition $E_{b}[3](\overline{\Q_p})=\chi\oplus\chi\cdot {\chi_{3,cyc}}$.

 Define $W^{\Q_p}_{b}\coloneq\{ (A,B) \in \mb{A}^2_{X}(\Q_p) : E_{A,B}[3](\overline{\Q_p}) \cong E_b[3](\overline{\Q_p}) \}$ {and $W_b=W_b^{\Q_p} \cap \Z_p^2$.} The following lemma generalizes Lemma \ref{integral covering split X(3)}. 
\begin{lemma} \label{parametrization Eb lemma}
Let $b=16p^{3\delta}\nu ^{3\ve}$ with $\delta,\ve\in\{0,1\}$.
		 Define a map $\Phi \colon \Z_p^2 \setminus \{{\Delta_\Phi = 0}\} \xrightarrow{} \mb{A}^2_{X} $ by
		\begin{align*}
		 \Phi(u,z)&=\pi_{E_b}\left(u,zp^{-\delta}\right)\\&=\left(- 2^{16} 3^5 p^{2\delta}\nu ^{6\ve} z^4 + 2^7 3^2 p^{2\delta}\nu ^{3\ve} u^3 z, - 2^{25} 3^6p^{3\delta}\nu ^{9\ve} z^6 - 2^{15}3^35 p^{3\delta}\nu ^{6\ve} u^3 z^3 + 2^4p^{3\delta}\nu ^{3\ve}u^6\right),
		\end{align*}
 where $\Delta_\Phi(u,z)$ is the discriminant of the Weierstrass equation corresponding to $\Phi(u,z)$. Then,
 we have $W_{b}=\Phi(\Z_p^2 \setminus \{\Delta_\Phi = 0\} )$, and the map $\Phi$ is a $K_p$-to-$1$ covering of $W_{b}$, where $K_p=24$ for $p\equiv 1\pmod 3$ and $K_p=2$ for $p\equiv 2 \pmod 3$.
	\end{lemma}
\begin{proof}
 By Lemma \ref{twisted X(3) jacobian lemma}, we have $W^{\Q_p}_{b}=\pi_{E_b}(\mb{A}^2_{X_{E_b}(3)}(\Q_p) )$, where the map $\pi_{E_{b}} \colon\mb{A}^2_{X_{E_b}(3)} (\Q_p) \xrightarrow{} \mb{A}^2_{X}$ is defined by
		\[
		\pi_{E_b}(u,v)= \left(- 2^{16} 3^5 p^{6\delta}\nu ^{6\ve} v^4 + 2^7 3^2 p^{3\delta}\nu ^{3\ve} u^3 v, - 2^{25} 3^6p^{9\delta}\nu ^{9\ve} v^6 - 2^{15}3^35 p^{6\delta}\nu ^{6\ve} u^3 v^3 + 2^4p^{3\delta}\nu ^{3\ve}u^6\right).
		\] 
 Note that we have ${\Phi(u,z)}=\pi_{E_b}(u,zp^{-\delta})=(p^{2\delta}\Phi_1(u,z),p^{3\delta}\Phi_2(u,z))$ where 
 \begin{align*}
 \Phi_1&=- 2^{16} 3^5\nu ^{6\ve} z^4 + 2^7 3^2\nu ^{3\ve} u^3 z, \\
 \Phi_2&=- 2^{25} 3^6\nu ^{9\ve} z^6 - 2^{15}3^35 \nu ^{6\ve} u^3 z^3 + 2^4\nu ^{3\ve}u^6.
 \end{align*}

 It is clear that we have $ \Phi(\Q_p^2)=\pi_{E_b}(\Q_p^2)$. Let $E_b[3](\overline{\Q_p})={\chi_1\oplus \chi_2}$.
 If $p\equiv 2\pmod 3$, then $\chi_1\neq \chi_2$ since $\chi_1\chi_2={\chi_{3,cyc}}\neq \Id$. The only {nonidentity} Galois invariant automorphism of $E_b[3](\overline{\Q_p})$ that preserves the Weil pairing is the multiplication by $-1$, so $K_p=2$. If $p\equiv 1\pmod 3$, then $\chi_1= \chi_2$ and so $K_p={\#\SL_2(\F_3)}=24$ by Lemma \ref{preimage lemma}. We conclude that the map $\pi_{E_b}$ is a $K_p$-to-1 covering of $W^{\Q_p}_{b}$ by the same argument as in Lemma \ref{integral covering split X(3)}. To finish the proof, it suffices to prove that the set of points $(u,z)$ in {$\Q_p^2 \setminus \{\Delta_\Phi=0\}$} for which we have $\Phi(u,z) \in \Z_p^2$ is equal to $\Z_p^2 \setminus \{\Delta_\Phi = 0\}$. 

 The resultant of $\Phi_1$ and $\Phi_2$ in $u$ is $-2^{108}3^{39} \nu ^{45 \ve} z^{24}$ and the one in $z$ is $-2^{112}3^{39} \nu ^{48 \ve} u^{24}$. Let $(u,z) \in \Q_p^2 \setminus \{\Delta_\Phi = 0\}$, and suppose that we have $\Phi(u,z)\in \Z_p^2$. If $v_p(u)>v_p(z)$, then $v_p(p^{2\delta}\Phi_1(u,z))=2\delta+4v_p(z)$. As $\delta\in \{0,1\}$, we have $v_p(z)\geq 0$. If $v_p(z)>v_p(u)$, a similar argument shows that $v_p(u)\geq 0$. In both cases we have $(u,z) \in \Z_p^2$, as desired.
 
 Assume now $v_p(u)=v_p(z)=-k$. Suppose first that $k >0$ holds. Clearing denominators, we write $u=u'p^{-k}$ and $z=z'p^{-k}$ with $v_p(u')=v_p(z')=0$. We then have $0 \leq v_p(p^{2\delta}\Phi_1(u,z))=2\delta-4k +v_p(\Phi_1(u',z'))$, and $0 \leq v_p(p^{3\delta}\Phi_2(u,z))=3\delta-6k +v_p(\Phi_2(u',z'))$. Since $0 \leq \delta \leq 1$
 and $k \geq 1$, $p$ divides both $\Phi_1(u',z')$ and $\Phi_2(u',z')$. So $p$ must divide both resultants and then $p$ divides $u'$ and $z'$, but this is impossible. Thus we have $k \leq 0$ and $(u,z) \in \Z_p^2$.
\end{proof}

	\begin{lemma}\label{lem:finalmeasure}
		If $p\equiv 1\pmod 3$, then
		\[
		\mu_{\Z_p^2}(W_b)=\mu_{\Z_p^2}(\Phi(\Q_p^2)\cap \Z_p^2)=\frac{p^{8-5\delta}(p-1)(p^4-3p^3+3p-1)}{24(p^{10}-1)(p^3-1)}.
		\]
		If $p\equiv 2\pmod 3$, then 
		\[
		\mu_{\Z_p^2}(W_b)=\mu_{\Z_p^2}(\Phi(\Q_p^2)\cap \Z_p^2)=\frac{p^{8-5\delta}(p - 1)^2 (p^3 +1)}{2(p^{10}-1)(p^3-1)}.
		\]
	\end{lemma}
 \begin{proof}
 The Jacobian of $\Phi$ is directly computed to be
		\[
		J(\Phi(u,z))=-(48)^3\nu ^{6\ve}p^{5\delta}u^2((48 \nu ^\ve z)^3+u^3)^2.
		\]
 The result now follows by applying Lemma \ref{p-adic integral lemma more general} and using that $\Phi$ is a $K_p$-to-1 covering.
 \end{proof}

\begin{lemma}\label{lem:u24}
 Assume $p\equiv 1\pmod 3$. Then $\eta_2=0$ and
 \[
 \eta_4=\frac{p^{3}(p-1)(p^4-3p^3+3p-1)(p^5+1)}{12(p^{10}-1)(p^3-1)}.
 \]
 Assume $p\equiv 2\pmod 3$. Then $\eta_4=0$ and
 \[
 \eta_2=\frac{p^3(p - 1)^2 (p^5+1) (p^3 +1)}{2(p^{10}-1)(p^3-1)}.
 \]
\end{lemma}
\begin{proof}
Assume $p\equiv 1\pmod 3$. Let $b=16p^{3\delta}\nu ^{3\ve}$ for $\delta,\ve\in\{0,1\}$. If $\delta=\ve=0$, then $E_b[3](\overline{\Q_p})=\chi_1\oplus \chi_1$. If $\delta=0$, $\ve=1$, then $E_b[3](\overline{\Q_p})=\chi_\nu \oplus \chi_\nu $. If $\delta=1$, $\ve=0$, then $E_b[3](\overline{\Q_p})=\chi_p\oplus \chi_p$. If $\delta=\ve=1$, then $E_b[3](\overline{\Q_p})=\chi_{\nu p}\oplus \chi_{\nu p}$. So, by \eqref{decomposition count isogeny},
\[
\eta_4=\sum_{\delta,\ve\in\{0,1\}} \mu_{\Z_p^2}(W_{16p^{3\delta}\nu ^{3\ve}}).
\]
The lemma now follows from Lemmas \ref{multiple isogeny lemma} and \ref{lem:finalmeasure}.

Assume $p\equiv 2\pmod 3$. Proceeding as in the previous case, $\eta_4=0$ and
\[
\eta_2=\sum_{\delta\in\{0,1\}} \mu_{\Z_p^2}(W_{16p^{3\delta}}).
\]
\end{proof}
Now, we need to compute $\eta_1$.
\begin{lemma}
 Fix $\lambda\in\{1,p,\nu ,\nu p\}$. Let $\Psi_{\lambda}\colon\Q_p^2\to \Q_p^2$ be defined by
 \[
 \Psi_{\lambda}(u,v)=\left(2\lambda u v-\frac{\lambda^2u^4}{3},\lambda v^2-\frac{2\lambda^2vu^3}{3}+\frac{2\lambda ^3u^6}{27}\right).
 \]
 An elliptic curve $y^2=x^3+Ax+B$ admits a $\Q_p$-rational $3$-isogeny with {the character $\chi_\lambda$ associated with $\Q_p(\sqrt{\lambda})/\Q_p$ (trivial when $\lambda=1$)} if and only if $(A,B)$ is in the image of $\Psi_{\lambda}$.
\end{lemma}
\begin{proof}
 Follows from Lemma \ref{lem:explchar}.
\end{proof}
\begin{lemma} \label{Integral twisted X_1(3) covering lemma}
 We have 
 $\Psi_{\lambda}^{-1}(\Z_p^2) \subset \Z_p^2$.
 
\end{lemma}
\begin{proof}
We need to show that if for $(u,v) \in \Q_p^2$ we have{$\Psi_{\lambda}(u,v)\in \Z_p^2$, then $(u,v) \in \Z_p^2$.}
 Assume $v_p(v)>3v_p(u)+v_p(\lambda)$. Then the valuation of the second coordinate is $3v_p(\lambda)+6v_p(u)\geq 0$ and then $v_p(u)\geq 0$ (since $v_p(\lambda)\leq 1$). So, $v_p(v)\geq 0$ and both $u$ and $v$ are integers.

 Assume $v_p(v)<3v_p(u)+v_p(\lambda)$. Then the valuation of the second coordinate is $v_p(\lambda)+2v_p(v)\geq 0$ and then $v_p(v)\geq 0$. So, $v_p(u)\geq 0$ and $u$ and $v$ are integers.

 Assume $v_p(v)=3v_p(u)+v_p(\lambda)$, put $v_p(u)=k$. If $k\geq 0$, then $u$ and $v$ are integers. It remains {to consider} the case $k<0$. Put $u'=up^{-k}$ and $v'=vp^{-3k}\lambda^{-1}$ and notice $v_p(u')=v_p(v')=0$. Moreover,
 \begin{align*}
 \Psi_{\lambda}(u,v)&=\left(2\lambda u v-\frac{\lambda^2u^4}{3},\lambda v^2-\frac{2\lambda^2vu^3}{3}+\frac{2\lambda ^3u^6}{27}\right)\\&=\left(\lambda^2p^{4k}\left(2 u' v'-\frac{(u')^4}{3}\right),\lambda^3p^{6k}\left( (v')^2-\frac{2v'(u')^3}{3}+\frac{2 (u')^6}{27}\right)\right).
 \end{align*}
 Since {$\Psi_{\lambda}(u,v)\in \Z_p^2$} and $k<0$, we have that $p$ divides both $(2 u' v'-(u')^4/3)$ and $( (v')^2-2v'(u')^3/3+2 (u')^6/27)$. 
 The resultant of the two polynomials in $u'$ and $v'$ is $-(v')^8/3^9$ and $-(u')^8/3^3$ and so $p$ must divide $u',v'$, contradiction. So, if $k<0$, $\Psi_{\lambda}(u,v)$ {does not lie in $\Z_p^2$}. 
\end{proof}
\begin{lemma}\label{lem:124}
 Let $\eta_\lambda$ be the measure of the image of $\Psi_\lambda(\Z_p^2)$. Then
 \[
 2\eta_1+4\eta_2+8\eta_4=\sum_{\lambda \in 
 \{1,\nu ,p,\nu p\}
 }\eta_\lambda.
 \]
\end{lemma}
\begin{proof}
 The idea is to count how many times a Weierstrass equation appears in the decomposition (\ref{decomposition count isogeny}), with the possibilities being $1,2$ and $4$ based on the number of $\Q_p$-rational 3-isogenies. Assume first that $p\equiv 2\pmod 3$. By Lemma \ref{multiple isogeny lemma}, every elliptic curve over $\Q_p$ has at most two different $\Q_p$-rational 3-isogenies, and $\eta_4=0$. Next, we show that the map $\Psi_\lambda \colon \Z_p^2 \xrightarrow{} \Z_p^2$ is $2$-to-$1$ onto its image. Consider $(a,b) \in \Psi_{\lambda}(\Z_p^2)$ that defines an elliptic curve $E_{a,b}$. If $E_{a,b}$ has a unique $\Q_p$-rational 3-isogeny (corresponding to $\lambda$), then by Proposition \ref{3-isogeny parametrization} and an argument similar to the one used in Lemma \ref{preimage lemma}, there are exactly two preimages $(u,v),(-u,-v)$ in $\Q_p^2$, which are actually inside $\Z_p^2$ by Lemma \ref{Integral twisted X_1(3) covering lemma}. If $E_{a,b}$ has multiple $\Q_p$-rational 3-isogenies, then by Lemma \ref{multiple isogeny lemma} it has two distinct 3-isogenies, {and} the same argument applied to each of the isogenies shows that each of them has two preimages in $\Z_p^2$ under $\Psi_{\lambda}$, and the lemma follows. The case $p\equiv 1\pmod 3$ is exactly the same with the added possibility of 4 different $\Q_p$-rational 3-isogenies.
\end{proof}
\begin{lemma}\label{lem:mulambda}
 Let $\lambda\in\{1,\nu ,p,\nu p\}$. We have $\eta_\lambda=\abs{\lambda}_p^2\frac{p^3-p^2}{p^3-1}$ and \[
 \sum_{\lambda \in 
 \{1,\nu ,p,\nu p\}
 }\eta_\lambda=\frac{2(p-1)(p^2+1)}{p^3-1}.
 \]
\end{lemma}
\begin{proof}
 The Jacobian of $\Psi_\lambda$ is $4v^2\lambda^2$ by Lemma \ref{lem:explchar} and so
 \[
 \eta_\lambda=\int_{\Z_p^2}\abs{4v^2\lambda^2}_p \,du \,dv=\abs{\lambda}_p^2\int_{\Z_p}\abs{v}_p^2\,dv=\abs{\lambda}_p^2\frac{p^3-p^2}{p^3-1}.
 \]
 The last equation follows from Lemma \ref{p-adic integral lemma more general}. So,
 \[
 \sum_{\lambda \in 
 \{1,\nu ,p,\nu p\}
 }\eta_\lambda=\frac{p^3-p^2}{p^3-1}\left(2+\frac{2}{p^2}\right)=\frac{2(p-1)(p^2+1)}{p^3-1}.
 \]
\end{proof}
We are now ready to prove Theorem \ref{thm:mainiso}. By Lemma \ref{multiple isogeny lemma}, $\mu_{\Z_p^2}(W_{\iso})=\eta_1+\eta_2+\eta_4$. We computed $\eta_2$, $\eta_4$, and $2\eta_1+4\eta_2+8\eta_4$, whence we just need to combine these results.
\begin{proof}[Proof of Theorem \ref{thm:mainiso}]
 Assume $p\equiv 1\pmod 3$. Then, by Lemma \ref{lem:124},
 \[
 \eta_1+\eta_2+\eta_4=(\eta_1+2\eta_2+4\eta_4)-\eta_2-3\eta_4=\frac{\sum_{\lambda \in 
 \{1,\nu ,p,\nu p\}
 }\eta_\lambda}{2}-\eta_2-3\eta_4=\frac{\sum_{\lambda \in 
 \{1,\nu ,p,\nu p\}
 }\eta_\lambda}{2}-3\eta_4.
 \]
 By Lemmas \ref{lem:u24} and \ref{lem:mulambda}, this is equal to
 \begin{align*}
 &\frac{(p-1)(p^2+1)}{p^3-1}-\frac{p^{3}(p-1)(p^4-3p^3+3p-1)(p^5+1)}{4(p^{10}-1)(p^3-1)}\\&=\frac{3p^4 + 3p^3 + 4p^2 + 4}{4(p^4 + p^3 + p^2 + p + 1)}.
 \end{align*}

 Assume now $p\equiv 2\pmod 3$. Then, by Lemmas \ref{lem:u24} and \ref{lem:mulambda}, we have
 \begin{align*}
 \eta_1+\eta_2+\eta_4&=\frac{\sum_{\lambda \in 
 \{1,\nu ,p,\nu p\}
 }\eta_\lambda}{2}-\eta_2\\&=\frac{(p-1)(p^2+1)}{p^3-1}-\frac{p^3(p - 1)^2 (p^5+1) (p^3 +1)}{2(p^{10}-1)(p^3-1)}\\&=\frac{p^4 + p^3 + 2p^2 + 2}{2(p^4 + p^3 + p^2 + p + 1)}.
 \end{align*}
\end{proof}

\section{\texorpdfstring{$p$}{}-torsion over \texorpdfstring{$\Q_p$}{}} \label{p-torsion section}

In Section \ref{sec:3tors}, we assumed $p>3$, so that the characteristic is different from the order of the torsion points we are searching for. Now, we want to count how many elliptic curves over $\Z_p$ have a $\Q_p$-rational non-trivial $p$-torsion point. We are going to make this computation explicit for $p=3$.
	Let $p\geq 3$. Let $E/\Q_p$ be defined by a Weierstrass equation with coefficients in $\Z_p$. From \cite[Proposition VII.2.2]{silverman2009arithmetic} and {\cite[Theorem IV.6.4]{silverman2009arithmetic}}, the map 
	\[
	\phi(P)={\log_E\left(-\frac{x(P)}{y(P)}\right)}
	\]
	from $E_1(\Q_p)$ to $p\Z_p$ is an isomorphism (the group law on $p\Z_p$ is the addition). Notice that $v_p(x(P)/y(P))>0$ if $P\in E_1(\Q_p)$ and $v_p(x(P)/y(P))=-v_p(x(P))/2$. So, $\phi(P)\in p^2\Z_p$ if and only if $v_p(x(P)/y(P))\geq 2$ and then $-v_p(x(P))\geq 4$ (since $\phi(P)\equiv-\frac{x(P)}{y(P)}\pmod{p^2}$) .
\begin{lemma}\label{prop:condpsi2}
 Let $E/\Q_p$ be an elliptic curve defined by a Weierstrass equation in minimal form.
There is a $\Q_p$-rational non-trivial $p$-torsion point $Q$ on $E$ such that $Q$ is non-singular in the reduction modulo $p$ if and only if the following hold: there exists $(x_0,y_0)\in (\Z/p^2\Z)^2$ such that $\psi_p(x_0)\equiv 0\pmod {p^2}$, the pair $(x_0,y_0)$ satisfies the Weierstrass equation modulo $p^2$ and $(x_0,y_0)$ reduces to a non-singular point modulo $p$.
\end{lemma}
\begin{proof}
		If there is a $\Q_p$-rational non-trivial $p$-torsion point $Q$, then we take the reduction modulo $p^2$ of the coordinates of $Q$ and easily conclude.
		
		Now, we prove the converse. Let $R\in E(\Q_p)$ be a lift of $(x_0,y_0)$, that exists since $(x_0,y_0)$ is non-singular modulo $p$, and notice $R\notin E_1(\Q_p)$ since $R$ does not reduce to the identity modulo $p$. Hence, $\psi_p(x(R))\equiv 0\pmod {p^2}$. We have
		\[
		-v_p(x(pR))=2v_p(\psi_p(x(R)))-v_p(\theta_p(x(R)))=2v_p(\psi_p(x(R)))\geq 4
		\]
		since $R$ reduces to a non-singular point (see \cite[Theorem 1.1]{voutier2021greatest}) and then $\phi(pR)\in p^2\Z_p$ (recall $\theta_p$ is defined in Section \ref{preliminaries}). 
		Let $P\in E_1(\Q_p)$ be such that $\phi(P)=\phi(pR)/p\in p\Z_p$, that exists since $\phi$ is an isomorphism. 
		So,
		\[
		\phi(p(P-R))=p\frac{\phi(pR)}{p}-\phi(pR)=0
		\]
		and then $p(P-R)=0_E$. Put $Q=P-R$ and notice $Q\neq 0_E$ since $P\in E_1(\Q_p)$ and $R\notin E_1(\Q_p)$. This is a non-trivial $\Q_p$-rational $p$-torsion point and its reduction is non-singular since $P$ and $R$ both reduce to a non-singular point.
\end{proof}

Assume that $E$ has split multiplicative reduction. Let $q\in \Q_p$ be such that the $j$-invariant of the Tate curve $E_q$ is equal to the $j$-invariant of $E$, which exists by \cite[Theorem C.14.1]{silverman2009arithmetic}. In particular, by \cite[Theorem C.14.1.b]{silverman2009arithmetic}, $E$ and $E_q$ are isomorphic over $\Q_p$. We have $v_p(q)=v_p(\Delta)=N>0$ and $E(\Q_p)$ is isomorphic to $\Q_p^*/q^{\Z}$ (see \cite[Page 444 and Theorem C.14.1]{silverman2009arithmetic}). Moreover, $q_1\equiv \Delta_1 \pmod {p^2}$, for $q=p^Nq_1$ and $\Delta=p^N\Delta_1$. We are going to use the well-known fact that the equation $x^p-x_1=0$ with $x_1\in \Z_p^*$ has a solution in $\Q_p$ if and only if it has a solution modulo $p^2$. 

 \begin{lemma}\label{lem:mult31}
		Let $E/\Q_p$ be an elliptic curve with split multiplicative reduction. There is a $\Q_p$-rational non-trivial $p$-torsion point $Q$ on $E$ such that $Q$ is singular modulo $p$ if and only if $p$ divides $N$ and $x^p\equiv \Delta_1\pmod {p^2}$ has a solution.
	\end{lemma}
	\begin{proof}
		Assume that there is a singular $\Q_p$-rational $p$-torsion point $Q$. Then $Q\notin E_0(\Q_p)$ and $pQ\in E_0(\Q_p)$. Let $u\in \Q_p^*/q^{\Z}$ be representing $Q$ under the isomorphism above. So, $u^p=q^a$ for some $a\in \Z$. Since $Q\notin E_0$, we have $p\nmid a$. Hence, $p\mid v_p(q^a)=Na$ and so $p\mid N$. Let $u=p^{aN/p}u_1$ with $u_1\in\Z_p^*$ and then $u_1^p=u^p/p^{aN}=q_1^a$. Let $b,r\in\Z$ be such that $ab=1+pr$. So, $q_1=u_1^{pb}q_1^{-pr}$ and then $q_1$ is a $p$-th power and we conclude that $\Delta_1$ is a $p$-th power modulo $p^2$.
		
		Assume now that $p\mid N$. Let $u_1$ be a solution of the equation $u_1^p=q_1$, that exists since $x^p\equiv \Delta_1\equiv q_1 \pmod {p^2}$ has a solution. Let $u=p^{N/p}u_1$ and we have $u^p=q$. The point corresponding to $u$ in the isomorphism is the $p$-torsion point we were searching for. Indeed $u\notin E_0(\Q_p)$ by \cite[Theorem C.14.1.(c)]{silverman2009arithmetic} noticing that $p\mid u$ and $u^p$ is the identity in $\Q_p^*/q^{\Z}$.
	\end{proof}
	\begin{lemma}\label{lem:mult3}
		Let $E/\Q_p$ be an elliptic curve with split multiplicative reduction. The density of elliptic curves such that $v_p(\Delta)=N$ and $x^p\equiv \Delta_1\pmod {p^2}$ has a solution is $(p-1)^2/2p^{N+3}$.
	\end{lemma}
	\begin{proof}
		As explained in \cite[proof of Proposition 2.5]{cremona2020localpub}, after a translation $\Delta\equiv -b_2^3a_6\pmod {p^{N+2}}$ with $p\nmid b_2$ and $p^N\mid\mid a_6$. Let $a_6=a_6'p^N$ for $a_6'\in\Z_p^*$ and then $\Delta_1\equiv -b_2^3a_6'\pmod {p^{2}}$. Once $b_2$ is fixed, the {density} of $a_6'\in\Z_p^*$ such that $x^p\equiv \Delta_1\pmod {p^2}$ has a solution is $1/p$. We conclude by \cite[Proposition 2.5]{cremona2020localpub}.
	\end{proof}
	\begin{lemma} \label{lem:nstor}
		Let $E/\Q_p$ be an elliptic curve with non-split multiplicative reduction. There is no $\Q_p$-rational non-trivial $p$-torsion point $Q$ on $E$ such that $Q$ is singular modulo $p$.
	\end{lemma}
	\begin{proof}
		As explained in {\cite[Proof of Corollary C.15.2.1]{silverman2009arithmetic}}, $E(\Q_p)/E_0(\Q_p)$ has order at most $2$ and then we cannot have a singular $\Q_p$-rational $p$-torsion point.
	\end{proof}

\subsection{\texorpdfstring{$3$}{}-torsion over \texorpdfstring{$\Q_3$}{}}
	Now, we are going to assume $p=3$ and prove Theorem \ref{thm:main} in this case.
 {We will compute separately the probabilities that $E$ has a non-trivial $\Q_3$-rational $3$-torsion point in the following cases: $E$ is non-singular modulo $3$ (Lemma \ref{lem:gr3}); $E$ is singular modulo $3$, but the point has non-singular reduction (Lemma \ref{lem:add3}); and the point has singular reduction. In the last case, non-split multiplicative reduction is treated in Lemma \ref{lem:nstor}, additive reduction in Lemma \ref{lem:IV3}, and split multiplicative reduction in Lemma \ref{lem:mult3}. We will then sum these probabilities.}
 \begin{lemma}\label{lem:gr3}
 Let $W_{\gr,3}$ be the set of triples $[a_2,a_4,a_6]\in {\Z_3^3}$ such that $y^2=x^3+a_2x^2+a_4x+a_6$ is in minimal form and {nonsingular modulo $3$}, and it has a non-zero non-singular $\Q_3$-rational $3$-torsion point.
 Then $\mu_{{\Z_3^3}}(W_{\gr,3})=2/27$.
 \end{lemma} 
 \begin{proof}
 Let $[\overline{a_2},\overline{a_4},\overline{a_6}]$ be the reduction modulo $9$ of $[a_2,a_4,a_6]$. By Lemma \ref{prop:condpsi2}, $y^2=x^3+a_2x^2+a_4x+a_6$ is in $W_{\gr,3}$ if and only {if} there exists $(x_0,y_0)\in (\Z/9\Z)^2$ such that $\psi_3(x_0)\equiv 0\pmod 9$, the pair $(x_0,y_0)$ satisfies the Weierstrass equation modulo $9$, and the Weierstrass equation is non-singular modulo $3$. 

 Let $\Tilde{E}$ be the reduction modulo $3$. If $\Tilde{E}$ does not have any non-zero $\F_3$-rational $3$-torsion point, then $E$ does not belong to $W_{\gr,3}$. Assume $\Tilde{E}$ has a non-zero $\F_3$-rational $3$-torsion point $\Tilde{P}$. After a change of variables, we can assume that $x(\Tilde{P})\equiv 0\pmod 3$ and so $4a_2a_6\equiv a_4^2\pmod{3}$ since $\psi_3(0)\equiv 0\pmod 3$. Moreover, $3\nmid a_6$ since otherwise the curve would be singular. The {value} $\psi_3(x(\Tilde{P}))$ is {congruent} to $4a_2a_6-a_4^2$ modulo $9$ and so, once $a_4$ and $a_6$ are fixed, we can lift $a_2$ with probability $1/3$.
 Of the $27$ Weierstrass equations $y^2=x^3+a_2x^2+a_4x+a_6$ modulo $3$, exactly $6$ are non-singular and have a non-trivial $\F_3$-rational $3$-torsion point. So, $\mu_{{\Z_3^3}}(W_{\gr,3})=(6/27)/3=2/27$.
 \end{proof}
 \begin{lemma}\label{lem:add3}
 Let $W_{\add,3}$ (resp. $W_{\mult,3}$) be the set of triples $[a_2,a_4,a_6]\in {\Z_3^3}$ such that $y^2=x^3+a_2x^2+a_4x+a_6$ is in minimal form and with additive reduction (resp. multiplicative reduction), and it has a non-zero non-singular $\Q_3$-rational $3$-torsion point.
 Then, $\mu_{{\Z_3^3}}(W_{\add,3})=1/27$ and $\mu_{{\Z_3^3}}(W_{\mult,3})=0$.
 \end{lemma} 
 \begin{proof}
Similar to the previous one. The Weierstrass equations in minimal form and with additive reduction have measure $1/9$, they always have a non-singular $3$-torsion point in $\F_3$, and the probability of lifting such a point is $1/3$. Weierstrass equations with multiplicative reduction do not have a non-singular $3$-torsion point in $\F_3$.
 \end{proof}
 
	Assume that $E$ is of type $IV$.
	After a change of variables, $E$ is defined by a minimal Weierstrass equation of the form $y^2=x^3+a_2x^2+a_4x+a_6$ with $3\mid a_2$, $3^2\mid a_4$, $3^2\mid a_6$, and $3^3\nmid a_6$ (see \cite[Section 7]{tatealg}). In the next lemma, we are going to assume this.

 In the same way, if $E$ is of type $IV^*$, then $E$ is defined, after a change of variables, by a minimal Weierstrass equation of the form $y^2=x^3+a_2x^2+a_4x+a_6$ with $3^2\mid a_2$, $3^3\mid a_4$, $3^4\mid a_6$, and $3^5 \nmid a_6$ (see again \cite[Section 8]{tatealg}).

\begin{lemma}\label{lem:IV3}\hfill
\begin{enumerate}
\item Assume that $E$ is of type $IV$. There is a $3$-torsion point $Q\in E(\Q_3)$ such that $Q$ reduces to a singular point $(0,0)$ modulo $3$ if and only if $9\mid a_2$ and $\frac{a_6}{3^2}\equiv 1\pmod{3}$.
\item Assume that $E$ is of type $IV^*$. There is a $3$-torsion point $Q\in E(\Q_3)$ such that $Q$ reduces to a singular point $(0,0)$ modulo $3$ if and only if $3^7\mid a_2a_6-a_4^2$ and $\frac{a_6}{3^4}\equiv 1\pmod{3}$.
\item The density of $E$ with additive reduction having a non-trivial $3$-torsion point $(x_0,y_0)\in E(\Q_3)$ reducing to $(0,0)$ modulo $3$ is $1/6$ for type $IV$ and $1/6$ for type $IV^*$.	
\end{enumerate}
\end{lemma}

\begin{proof}
We first prove (1). Assume that $(3x,3y)\in E(\Q_3)$ is a $3$-torsion point such that $x,y\in\Z_3$. Write $a_2=3a'_2$, $a_4=3^2a'_4$, $a_6=3^2a'_6$, where $a'_2,a'_4\in \Z_3$ and $a'_6\in \Z_3^*$. Then $3^3\mid (3x)^3+3a'_2(3x)^2+3^2a'_4(3x)$, so we have the congruence $y^2\equiv a'_6 \pmod{3}$, and then $a'_6\equiv 1\pmod{3}$. By definition of $\psi_3$, 
\[
\psi_3(3x)=3^5x^4+3^4\cdot4a'_2x^3+3^5\cdot 2a'_4x^2+3^4\cdot 4a'_6x+3^3\cdot 4a_2'a'_6{-3^4{a'_4}^2}.
\]
Since ${\psi_3}(3x)=0$, we have $3^4\mid 3^3\cdot 4a_2'a'_6$, so $3\mid a'_2$ and we are done. 

Now, we prove the converse. Assume $a'_2=3a''_2$, for $a''_2\in \Z_3$ and $a_6'\equiv 1\pmod{3}$. Then, 
\[
\dfrac{\psi_3(3x)}{3^4}=3x^4+3\cdot4a_2''x^3+3\cdot 2a'_4x^2+ 4a'_6x+ 4a_2''a'_6{- {a'_4}^2},
\]
which reduces to a linear polynomial modulo $3$, hence has a solution $x\in \Z_3$ by Hensel's lemma. Moreover, $3^3\mid (3x)^3+3a'_2(3x)^2+3^2a'_4(3x)$, so we have the congruence $y^2\equiv a'_6 \pmod{3}$, that has a solution since $a'_6\equiv 1\pmod{3}$ and we can lift it to $y\in \Z_3$. This gives the point we were searching for.

Now we prove (2). Assume that $(3x,3y)\in E(\Q_3)$ is a $3$-torsion point such that $x,y\in\Z_3$. Write $a_2=3^2a'_2$, $a_4=3^3a'_4$, $a_6=3^4a'_6$, where $a'_2,a'_4\in \Z_3$ and $a'_6\in \Z_3^*$. If $3\nmid x$, then 
\[
v_3((3x)^3+3^2a'_2(3x)^2+3^3a'_4(3x)+3^4a'_6)=3,
\] hence $y\notin \Q_3$. So, we have $3\mid x$, and we may assume that $(3^2x,3^2y)\in E(\Q_3)$ is a $3$-torsion point (replacing $(x,y)$ with $(3x,3y)$). The equation of $E$ divided by $3^4$ gives $y^2\equiv a'_6\pmod{3}$, so again, we have $a'_6\equiv 1\pmod{3}$. Moreover,
\[\psi_3(3^2x)=3^9x^4+3^8\cdot4a'_2x^3+3^8\cdot 2a'_4x^2+3^7\cdot 4a'_6x+3^6\cdot 4a_2'a'_6{-3^6{a'_4}^2}.\]
Since $\psi_3(3^2x)=0$, we have $3\mid a_2'a_6'-{a'_4}^2$.

Now, we prove the converse. Assume $3\mid a_2'a_6'-{a'_4}^2$, so $\frac{\psi_3(3^2x)}{3^7}$ reduces to a linear polynomial modulo 3, hence, it has a root by Hensel's lemma. Once again we can find $y\in \Z_3$ such that $(3^2x,3^2y)$ is the $3$-torsion point we were searching for since $a'_6\equiv 1\pmod{3}$.

Finally for (3), the density {in} (1) is clear, and the density in (2) follows from the fact that $3\nmid a'_6$, so $a'_2$ is uniquely determined modulo $3$ by $a_4'$ and $a_6'$.
\end{proof}
\begin{remark}
 One may wonder why in the case $IV$ and $IV^*$ we have to look at the equation modulo a higher power of $3$ (compared to the case {of} good reduction, where we were {working} modulo $p^2$). This follows from the fact that for a singular point $P$, the valuation of $\gcd(\psi_3^2(x(P)),\theta_3(x(P)))$ is $6$ or $12$ (see \cite[Table 1.1]{voutier2021greatest}) and so to see if $x(3P)^{-1}\equiv 0\pmod 3$ (that is, $P$ is $3$-torsion modulo $3$) one has to check if $\psi_3(x(P))$ is $0$ modulo $3^4$ or $3^7$.
\end{remark}
Finally, we prove Theorem \ref{thm:main} for $p=3$.
 \begin{theorem}\label{thm:3torsQ3}
 The probability that $E/\Q_3$ has a $\Q_3$-rational non-trivial $3$-torsion point is $\frac{3}{26}$.
\end{theorem}
\begin{proof}
 The probability of having {a $\Q_3$-rational 3-torsion point with non-singular reduction mod $3$} is $2/27+1/27$ by Lemmas \ref{lem:gr3} and \ref{lem:add3}. {By Lemma \ref{lem:nstor}, the only cases when the reduction of a $\Q_3$-rational 3-torsion point is singular are additive reduction and split multiplicative reduction.} The probability that $E$ is in minimal form and has additive reduction and a singular $\Q_3$-rational $3$-torsion point is $(3^{-5}+3^{-8})/3$ by Lemma \ref{lem:IV3} and \cite[{Proposition} 2.2]{cremona2020localpub}. Notice that we cannot have at the same time a singular and non-singular non-trivial $\Q_3$-rational $3$-torsion point because we cannot have full torsion. The probability that $E$ is in minimal form, has split multiplicative reduction, and a singular $\Q_3$-rational $3$-torsion point is 
 \[
 \sum_{k\geq 1}\frac{2}{3^{3k+3}}=\frac{2}{3^{6}}\frac{3^3}{3^3-1}=\frac{2}{3^3(3^3-1)}
 \]
 by Lemmas \ref{lem:mult31} and \ref{lem:mult3}. If $E$ has any other reduction type, there is no singular $3$-torsion. In conclusion, we get
 \[
 \frac{3^{10}}{3^{10}-1}\left(\frac{2}{27}+\frac{1}{27}+\frac{1}{3^{6}}+\frac{1}{3^{9}}+\frac{2}{3^3(3^3-1)}\right)=\frac{3}{26}.
 \]
\end{proof}

\subsection{Good ordinary reduction}

Let $E/\F_p$ be an elliptic curve over $\F_p$, $p \geq 3$, with ordinary reduction, {such that $E[p](\F_p)\cong\Z/p\Z$}. The goal of this section is to compute the probability that $E$ can be lifted to an elliptic curve over $\Z_p$ with a $\Q_p$-rational non-trivial $p$-torsion point.

Associated with $E$ there is a canonical lifting of $E$ to an elliptic curve $E_{\can}$ over $\Q_p$, which is uniquely characterized by the following properties: the modulo $p$ reduction of $E_{\can}$ is $E$, and reduction modulo $p$ induces an isomorphism $\mathrm{End}(E_{\can}) \xrightarrow{} \mathrm{End}(E)$ of {endomorphism} rings. This was first observed by Deuring \cite{deuring1941typen}, and later generalized to abelian varieties by Lubin-Serre-Tate \cite{lubin1964elliptic}.

\begin{remark}
To define rigorously the {``reduction modulo $p$''} map $\mathrm{End}(E_{\can}) \xrightarrow[]{} \mathrm{End}(E)$, consider the {N\'eron} model $\mathcal{E}/\Z_p$ of $E_{\can}$, a smooth group scheme over $\Z_p$ with generic fiber $E_{\can}/\Q_p$. By the universal property of {N\'eron} models, we can extend any endomorphism of $E_{\can}$ uniquely to an endomorphism of $\mathcal{E}$. We have a canonical bijection $\mathrm{End}(E_{\can}) \xrightarrow[]{} \mathrm{End}(\mathcal{E})$, which we compose with the map $\mathrm{End}(\mathcal{E}) \xrightarrow[]{} \mathrm{End(E)}$ obtained by base change to special fiber. \end{remark}

\begin{lemma}
 If $E/\F_p$ has a non-zero $p$-torsion point defined over $\F_p$, then the canonical lift $E_{\can}$ has a non-zero $p$-torsion point defined over $\Q_p$.
\end{lemma}
\begin{proof}
 Recall that the Verschiebung $V\colon E \xrightarrow[]{} E$ is a degree $p$ isogeny defined as the dual of the Frobenius $F\colon E \xrightarrow[]{} E$, and that $F \circ V=V \circ F=[p]$ is the multiplication-by-$p$ morphism. Since the {reduction} map $\mathrm{End}(E_{\can}) \xrightarrow{} \mathrm{End}(E)$ is an isomorphism, both $V$ and $F$ lift to degree $p$-isogenies $V_{\can}\colon E_{\can}\xrightarrow{} E_{\can}$ and $F_{\can} \colon E_{\can}\xrightarrow[]{}E_{\can}$ that satisfy $F_{\can} \circ V_{\can}=V_{\can} \circ F_{\can}=[p]$. Let $\mathcal{V} \colon\mathcal{E} \xrightarrow[]{} \mathcal{E}$ be the extension of $V_{\can}$ to an endomorphism of the {N\'eron} model $\mathcal{E}$ of $E_{\can}$. Then $\mathrm{ker}( \mathcal{V})$ is a scheme over $\Z_p$ with special fiber $\mathrm{ker}(V)$ and generic fiber $\mathrm{ker}(V_{\can})$. This implies that $\mathrm{ker}(\mathcal{V})$ is smooth (even étale) over $\mathrm{Spec} \ \Z_p$: consider the module of {K\"ahler} differentials $\Omega_{\mathrm{Ker}(\mathcal{V})/\mathrm{Spec \Z_p}}$. Since $\mathrm{ker}(V_{\can})/\mathrm{Spec} \ \Q_p$ and $\mathrm{ker}(V)/\mathrm{Spec} \ \F_p$ are smooth of relative dimension 0, we have $\Omega_{\mathrm{Ker}(\mathcal{V})/ \Z_p}\otimes_{\Z_p}\Q_p=\Omega_{\mathrm{Ker}(V_{\can})/\Q_p}=0$ and $\Omega_{\mathrm{Ker}(\mathcal{V})/ \Z_p}\otimes_{\Z_p}\F_p=\Omega_{\mathrm{Ker}(V)/\F_p}=0$, and so by Nakayama's lemma $\Omega_{\mathrm{Ker}(\mathcal{V})/\Z_p}=0$ and $\mathrm{Ker}(\mathcal{V})/\Z_p$ is smooth by \cite[Proposition III.10.4]{hartshorne2013algebraic} with $n=0$. By Hensel's lemma \cite[Corollary 6.2.13]{qingliu}, every $\F_p$-rational point of $\mathrm{ker}(V)$ {lifts} to a $\Z_p$-point of $\mathrm{ker}( \mathcal{V})$. By specializing to the generic fiber we deduce that $\mathrm{ker}(V_{\can})\cong \Z/p\Z$, and so we also have $E_{\can}[p](\Q_p)=\Z/p\Z$.
\end{proof}

We will need the following criterion of Serre for the field $\Q_p(E[p])$ to be tamely ramified. For a proof, see Theorem 2 of \cite{nakamura1993note}.
\begin{theorem} \label{Serre criterion}
 Let $E/\Q_p$ be an elliptic curve with good ordinary reduction. Let $j_{\can} \in \Q_p$ be the $j$-invariant of the canonical lift of the reduction of $E$ modulo $p$. Then the extension $\Q_p(E[p])/\Q_p$ is tamely ramified if and only if {$j \equiv j_{\can} \pmod{p^{\tau+1}}$, where $\tau=1$ if $j_{\can} \ne 0,1728$, $\tau=2$ if $j_{\can}=1728$, and $\tau=3$ if $j_{\can}=0$}.
\end{theorem}

{
\begin{lemma} \label{lem:serre-criterion}
 Let $p \geq 5$ be a prime and let $E/\Q_p$ be an elliptic curve with good ordinary reduction $\tilde{E}/\F_p$ such that $ \tilde{E}[p](\F_p)=\Z/p\Z$. Then the following are equivalent:
 \begin{enumerate}[label=(\roman*)]
 \item There is a non-trivial $\Q_p$-rational $p$-torsion point on $E$.
 \item We have $E[p](\bar{\Q_p}) \cong \mu_p \times \Z/p\Z$ as $\Gal(\bar{\Q}_p/\Q_p)$-modules.
 \item The field $\Q_p(E[p])$ is a tamely ramified extension of $\Q_p$.
 \end{enumerate}
\end{lemma}
\begin{proof}
Let $\tilde{E}$ denote the reduction of $E$ mod $p$. We have the exact sequence of $\Gal(\bar{\Q}_p/\Q_p)$-modules
\[
0 \to T_p \to E[p] \to \tilde{E}[p] \to 0
\]
where $T_p$ is the canonical subgroup of $E[p]$ obtained as the kernel of reduction mod $p$. Since $T_p(\Q_p)=E[p] \cap E_1(\Q_p)$ and $E_1(\Q_p) \cong \Z_p$, we deduce that $T_p(\Q_p)=\{0_E\}$. By our assumption $\tilde{E}[p] \cong \Z/p\Z$ is a trivial $\Gal(\bar{\Q}_p/ \Q_p)$-module, and since $\det E[p] \cong \mu_p$ by the Weil pairing, we have $T_p \cong \mu_p$ as Galois modules. It follows that any non-zero $\Q_p$-rational $p$-torsion point of $E$ defines a splitting of the exact sequence and hence an isomorphism $E[p] \cong \mu_p \times \Z/p\Z$, and so (i) and (ii) are equivalent. Since $\mu_p$ is tamely ramified (ii) implies (iii). We now show that (iii) implies (ii). Suppose that $\Q_p(E[p])$ is tamely ramified, and let $I$ and $I_p$ denote the inertia group and wild inertia group of $\Q_p$. Then $I_p$ acts trivially on $E[p]$, and so by \cite[Corollaire (b), pg 274]{serre1972proprietes}, we can choose a basis $e_1,e_2$ of $E[p]$ with $e_1 \in T_p$ such that the action of $I$ on $E[p]$ is described by matrices of the form $\begin{pmatrix}
 *&0\\0&1
\end{pmatrix}$. In particular $I$ fixes $e_2$, and as $T_P^{I} \cong \mu_p^{I}$ is trivial, it follows that $E[p]^{I}=\langle e_2 \rangle$. Hence the reduction mod $p$ map $E[p]^{I} \to \tilde{E}[p]$ is an isomorphism of $\mathrm{Gal}(\bar{\Q}_p/\Q_p)$-modules, and in particular $I$ fixes $e_2$ and $e_2 \in E[p](\Q_p)$.
\end{proof}}

 \begin{proposition} \label{prob lift of p torsion}
 Let $p\geq {5}$. Let $E/\F_p$ be an elliptic curve over $\F_p$ with an $\F_p$-rational non-trivial $p$-torsion point, defined by a fixed Weierstrass equation $y^2=x^3+ax+b$ for some $a,b \in \F_p$. The probability that a random lift of $E$ to $\Q_p$ has a $\Q_p$-rational non-trivial $p$-torsion point is $1/p$, and the Weierstrass equation of any such lift is congruent modulo $p^2$ to a Weierstrass equation of the canonical lift. More precisely, 
 \[
 \frac{\mu_{\Z_p^2}\left(\left\{A,B \in \Z_p^2 : A \equiv a \pmod{p}, B \equiv b \pmod{p}, \#\left(E_{A,B}[p](\Q_p)\right)=p\right\}\right)}
 {\mu_{\Z_p^2}\left(\{A,B \in \Z_p^2 : A \equiv a \pmod{p}, B \equiv b \pmod{p}\}\right)}=\frac{1}{p}
 \] 
 where $\mu_{\Z_p^2}$ is the Haar measure on $\Z_p^2$ {normalized so that $\mu_{\Z_p^2}(\Z_p^2)=1$}.
 Moreover, the set 
 \[
 \left\{A,B \in \Z_p^2 : A \equiv a \pmod{p}, B \equiv b \pmod{p}, \#\left(E_{A,B}[p](\Q_p)\right)=p\right\}
 \]
 is defined by congruence conditions modulo $p^2$.
 \end{proposition}

 We already know that Proposition \ref{prob lift of p torsion} holds for $p=3$, as shown in the proof of Lemma \ref{lem:gr3}, so we assumed $p\geq 5$.

\begin{proof}
{Let $E^{can}/\Q_p : y^2=x^3+Ax+B$ be the canonical lift of $E$.  Let $E': y^2=x^3+A'x+B'$ be any other lift of $E$.  Then by Lemma \ref{lem:serre-criterion}, we have $E'[p](\Q_p)=\Z/p\Z$ if and only if  $j^{can}  \equiv j' \pmod{p^{\tau+1}}$, with $\tau$ as defined in Theorem \ref{Serre criterion}. 

Assume $j\neq 0,1728$. Then $j^{can}  \equiv j' \pmod{p^{2}}$ if and only if
$u^2A\equiv A'\pmod{p^2}$ and
$u^3B\equiv B'\pmod{p^2}$, and as $a\ne0, b\ne 0$, necessarily
$u\equiv1\pmod p$. As there are $p^2$ possible lifts of $(a,b)$ to $(\Z/p^2\Z)^2$ and there are $p$ such $u \in \Z/p^2\Z$, we conclude that the probability is
$p/p^2=1/p$.

When $j=0$, as $\End (E^{can})=\End(E)$, $E^{can}$ also has an automorphism of order 3 and we can write $E^{can} : y^2=x^3+B$. Then $j \equiv j' \pmod{p^4}$ if and only if $1728\frac{4A^{'3}}{4A^{'3}+27B^{'2}} \equiv 0 \pmod{p^4}$, and as $4A^{'3}+27B^{'2}$ is a unit, this holds if and only if $A' \equiv 0 \pmod{p^2}$, and as this happens with probability $1/p$ among $A' \equiv 0 \pmod{p}$, we are done.

The case $j=1728$ is similar. Since $\End(E^{can})=\End{E}$ has an automorphism of order 4, we can assume that $E^{can}:y^2=x^3+Ax$. Then $j(E')\equiv j(E_{\can})\pmod{p^3}$ holds if and only if $B'\equiv0\pmod{p^2}$, and this happens with probability $1/p$, finishing the proof. 

All the conditions obtained depend only on the
coefficients modulo $p^2$.}

\end{proof}

\begin{remark}\label{rem:p-torsion-div-poly}
We first noticed the probability $1/p$ in the good reduction case by considering the division polynomial $\psi_p$ for fixed values of the prime $p$. By a linear change of variables, we may assume that $x=0$ is the $x$-coordinate of a non-zero $p$-torsion point in $E(\F_p)$. We observed that the coefficient of $x$ in $\psi_p$ is divisible by exactly $p$, regardless of $a_2,a_4,a_6\in\Z_p$. We similarly see that the constant coefficient is divisible by $p$, and mod $p^2$ is given by a linear form $l_p$ in $a_4,a_6\in\Z_p$, the coefficients of which depend only on the prime $p$. Combining these observations and Hensel's lemma, the root $x=0$ of $\psi_p$ modulo $p$ lifts to a root of $\psi_p$ in $\Z_p$ if and only if the constant coefficient is divisible by $p^2$, which happens with probability $1/p$. However, we have not managed to find a direct proof of these properties of the division polynomial $\psi_p$ for a general prime $p$, since it is very difficult to compute the coefficients of $\psi_p$ directly.
\end{remark}

{
\begin{remark}
We make the calculation explicit for $p=5$. Let
$E\colon y^2=x^3+a_2x^2+a_4x+a_6$, with
$a_2,a_4,a_6\in\Z_5$, have good reduction and suppose that
$(0,\overline{y})$ is a non-zero $5$-torsion point on its
reduction modulo $5$. A direct
calculation gives the constant and linear coefficients of
$\psi_5$:
\[
\begin{aligned}
c_0&
 =16a_6^2(2a_4(4a_2a_6-a_4^2)-16a_6^2)-(4a_2a_6-a_4^2)^3,\\
c_1&
 =-20a_6(16a_2^2a_6^2-24a_2a_4^2a_6
 +5a_4^4+32a_4a_6^2).
\end{aligned}
\]
If $E$ has good reduction and $(0,\overline{y})$ is a non-zero $5$-torsion point on its
reduction modulo $5$, the possible reductions of $(a_2,a_4,a_6)$ modulo $5$ are
$(1,0,1)$, $(2,1,1)$, $(2,1,4)$, $(3,1,1)$,
$(3,1,4)$, and $(4,0,4)$. For each of these reductions,
we have $5\mid c_0$, $5\mid\mid c_1$, and
$\partial c_0/\partial a_6\not\equiv0\pmod5$.
Moreover, the reduction of $\psi_5(x)$ modulo $5$
belongs to $\F_5[x^5]$.

We claim that the root $x=0$ modulo $5$ lifts to a root
of $\psi_5$ in $5\Z_5$ if and only if $25\mid c_0$.
Indeed, $\psi_5(5X)\equiv c_0\pmod{25}$, so this
condition is necessary. Conversely, if $25\mid c_0$,
then $g(X)=\psi_5(5X)/25$ belongs to $\Z_5[X]$ and
$g(X)\equiv c_0/25+(c_1/5)X\pmod5$. Since $c_1/5$
is a unit, Hensel's lemma gives a unique root of $g$
in $\Z_5$. The corresponding $y$-coordinate lifts
$\overline{y}$ by Hensel's lemma, since
$\overline{y}\neq0$. Thus, we obtain a $\Q_5$-rational
non-trivial $5$-torsion point.

Since $\partial c_0/\partial a_6$ is a unit, for fixed
$a_2,a_4$ and a fixed residue of $a_6$ modulo $5$,
exactly one of the five lifts of $a_6$ modulo $25$
satisfies $25\mid c_0$. Hence, the probability is $1/5$.
\end{remark}
}

\section{Complements}\label{sec:comm}

 \subsection{A sieve argument}\label{sieve section}
The goal of this section is to show that local information can help in counting elliptic curves with non-trivial $3$-torsion and with coefficients in a fixed interval. Let $E$ be an elliptic curve defined over $\Q$. There exists a unique Weierstrass equation $y^2=x^3+Ax+B$ defining $E$ and with $\gcd(A^3,B^2)$ 12-th power-free. Define the height of $E$ as $H(E)=\max\{\abs{A}^3,\abs{B}^2\}$. In \cite[Theorem 5.5]{harron2017counting}, Harron and Snowden showed that {the number of} elliptic curves with a $3$-torsion point of naive height up to $X$ {is} {$ c_3X^{1/3}+O(X^{1/4})$}, with $c_3>0$. We focus on a strictly related problem. Let $M_1, N_1,M_2,N_2$ {be} four integers with $N_1,N_2>0$. Let
\begin{equation}\label{eq:S}
 {S=\left\{(A,B)\in \Z^2 : 
 \begin{array}{l}
 M_1\leq A\leq M_1+N_1,
 \\ M_2\leq B\leq M_2+N_2, 
 \end{array}
 E\colon y^2=x^3+Ax+B \text{ is smooth }, E[3](\Q)\neq \{0_E\}\right\}.}
\end{equation}
By applying the large sieve, we establish an upper bound for the cardinality of the set $S$. For the box {$[-X^{1/3},X^{1/3}] \times [-X^{1/2},X^{1/2}]$} that essentially counts Weierstrass equations of naive height less than $X$, our result is weaker than \cite[Theorem 5.5]{harron2017counting} that even gives an asymptotic formula for the number of curves with a 3-torsion point. However, their method uses the Principle of Lipschitz and is based on using geometry of numbers to count lattice points in a bounded region, and only works well when the shape of this region is regular in a certain sense. Extending their method to treat boxes of arbitrary shape and center would require a rather delicate analysis. Using the large sieve we will now see that it is straightforward to prove an estimate showing that there are very few elliptic curves with a rational 3-torsion point in an arbitrary box. 
\begin{proposition} \label{sieve bound}
{Let $S$ be as in \eqref{eq:S}. Then, 
 \[
		\#S \ll \frac{N_1N_2}{\sqrt{\min\{N_1,N_2\}}}.
		\]}
\end{proposition}
Note that this improves on the trivial bound $\#S\ll N_1N_2$.
\begin{proof}
 Let $L=\sqrt{\min\{N_1,N_2\}}$ and, for all $p\leq L$, define 
 \[
 Y_p=\left\{(A,B)\in(\Z/p\Z)^2: E\colon y^2=x^3+Ax+B {\text{ is smooth and }} E[3](\F_p)=\{0_E\}\right\}.
 \]
 If $(A,B)\in S$, then the reduction modulo $p$ of $(A,B)$ does not belong to $Y_p$ for all $p\leq L$. Thus, by an application of the large sieve, in the form proved in \cite[Lemma 5.1]{cameron}, it holds
 \[
 \#S\ll \frac{(\sqrt{N_1}+L)^2(\sqrt{N_2}+L)^2}{F(L)}
 \]
 with 
 \[
 F(L)=\sum_{n\leq L}\mu^2(n)\prod_{p\mid n}(\delta_p/(1-\delta_p))
 \]
 for $\delta_p=\#Y_p/p^2$. To conclude, we just need to show that $F(L)\gg L$. By Lemma \ref{good reduction count}, $\delta_p\geq 1/2$ for all $p$ and then
 \[
 F(L)=\sum_{n\leq L}\mu^2(n)\prod_{p\mid n}(\delta_p/(1-\delta_p))\gg
 \sum_{n\leq L}\mu^2(n)\gg L.
 \]
 \end{proof}
The techniques of geometry of numbers that \cite{harron2017counting} employ give a better estimate than Proposition \ref{sieve bound} for elliptic curves because the large sieve is far from optimal for this situation. Efthymios Sofos has pointed out to us that it has much more success in the setting of short intervals, a topic that has recently been of considerable interest, see for example \cite{matomaki2016multiplicative}. If the interval is of the form $[J^2-J^{b},J^2]\times [J^3-J^{c},J^3]$, {for any $0<b<2$, $0<c<3$}, we note that the trivial bound is $O(J^{b+c})$, so if the intervals are too small, then this is already better than $O(J^2)$. If, for example, $b=c\in (0,\frac43)$, the large sieve gives a new result. We record the estimates we obtain in both of these settings below.
 \begin{corollary}\label{cor:sieve-cases} Assume the setting of Proposition \ref{sieve bound} and that $J$ is large enough. 
\begin{enumerate} 
\item If $M_1=M_2=0$, $N_1=J^2$, $N_2=J^3$, then $\#S\ll J^4$.
\item Assume $M_1=J^2-J^b$, $M_2=J^3-J^c$, $N_1=J^b$, $N_2=J^c$ for $0<b<2$, $0<c<3$. Then $\#S\ll J^{b+c-\frac{\min(b,c)}{2}}$.
\end{enumerate} 
\end{corollary}

	\subsection{Asymptotic probability of \texorpdfstring{$\Q_p$}{}-rational \texorpdfstring{$\ell$}{}-torsion }\label{subsec:l-torsion}
In this section we give an asymptotic result for $\Q_p$-rational non-trivial $\ell$-torsion points {on elliptic curves} as the prime $p \to \infty$.
	Let $\ell$ be an odd prime.
 Let $E/\F_p$ be an elliptic curve for $p\neq \ell$. Then, $\Frob_p$ acts on $E[\ell](\overline{\F_p})$, so we can see it as a matrix in $\GL_2(\F_\ell)$. We have $\det(\Frob_p)\equiv p\pmod \ell$ and $\Tr(\Frob_p)\equiv p+1\pmod \ell$, if $\ell\mid \#E(\F_p)$ since {$\#E(\F_p)\equiv p+1-\Tr(\Frob_p)\pmod\ell$}.
 Then, by the geometric version of Chebotarev density theorem, we have
 		\[
		\lim_{p\to\infty}\abs{\frac{\#\{E/\F_p \text{ s.t. } \ell\mid \#E(\F_p)\}/\sim}{\#\{E/\F_p \}/\sim}-\frac{\#\{M\in \GL_2(\F_\ell)\mid \det(M)= p, \Tr(M)=p+1\}}{\#\{M\in \GL_2(\F_\ell)\mid \det(M)=p\}}}=0,
		\]
 and so, fixing $s\in \Z$ coprime with $\ell$,
	\[
	\lim_{\substack{p\to \infty\\ p\equiv s\pmod \ell}}\frac{\#\{E/\F_p \text{ s.t. } \ell \mid \#E(\F_p)\}/\sim}{\#\{E/\F_p \}/\sim}= \frac{\#\{M\in \GL_2(\F_\ell)\mid \det(M)= s, \Tr(M)=s+1\}}{\#\{M\in \GL_2(\F_\ell)\mid \det(M)=s\}}.
	\]
 
	Let $W_{\ell,p}$ be the subset of $\Z_p^5$ of those $[a_1,a_2,a_3,a_4,a_6]$ for which $y^2+a_1xy+a_3y=x^3+a_2x^2+a_4x+a_6$ is an elliptic curve and has a non-trivial $\Q_p$-rational $\ell$-torsion point. Since the measure of coefficients such that the Weierstrass equation reduces to a singular curve modulo $p$ goes to $0$ {as} $p$ goes to infinity, we have
	\begin{equation}\label{eq:limgamma}
		\lim_{\substack{p\to \infty\\ p\equiv s\pmod \ell}}\mu_{\Z_p^5}(W_{\ell,p})=\lim_{\substack{p\to \infty\\ p\equiv s\pmod \ell}}\frac{\#\{E/\F_p \text{ s.t. } \ell \mid \#E(\F_p)\}/\sim}{\#\{E/\F_p \}/\sim}=\frac{\gamma_{s,s+1}}{\sum_{0\leq t\leq \ell-1}\gamma_{s,t}}
	\end{equation}
	where 
	\[
	\gamma_{s,t}=\#\{M\in \GL_2(\F_\ell)\mid \det(M)= s, \Tr(M)=t\}.
	\]
	Notice that
	\[
	\sum_{0\leq t\leq \ell-1}\gamma_{s,t}=\#\{M\in \GL_2(\F_\ell)\mid \det(M)= s\}=\#\SL_2(\F_\ell)=\ell^3-\ell.
	\]
	\begin{lemma}\label{lemma:countmat}
		Let $\legendre{\cdot}{\cdot}$ be the Legendre symbol. Then
		\[
		\gamma_{s,t}=\ell^2+\legendre{t^2-4s}{\ell} \ell.
		\]
		In particular,
		\[
		\gamma_{s,s+1}=\begin{cases}
			\ell^2 \mbox{ if } s\equiv 1\pmod \ell, \\
			\ell^2+\ell \mbox{ otherwise.}
		\end{cases}
		\]
	\end{lemma}
	\begin{proof}
		Let $M=\begin{pmatrix}
			a &b \\
			c & d 
		\end{pmatrix}$. Then we have to count $a,b,c,d\in \F_\ell$ such that $ad-bc=s$, $a+d=t$. If we fix $b$ and $c$, we have two choices for $a$ and $d$ if $t^2-4bc-4s$ is a non-zero square, one choice if it is zero, and zero choices if it is not a square. Assume that $t^2-4s$ is not a square. Then, if $bc=0$, $t^2-4bc-4s$ is not a square. If $4bc=t^2-4s$, then we have $1$ choice for $a$ and $d$. In this case, we have $\ell-1$ choices. We have $(\ell-1)/2$ choices for $bc\neq 0$ such that $t^2-4bc-4s$ is a non-zero square. For all of them, we have two choices for $a$ and $d$. Since we have $(\ell-1)^2/2$ choices for $b$ and $c$, we obtain $(\ell-1)^2$ choices. Adding them to the previous case, we obtain in total $(\ell-1)^2+\ell-1=\ell^2-\ell$. The other cases are similar.
		
		If $s\equiv 1\pmod \ell$, then $(s+1)^2-4s\equiv 0\pmod \ell$ and then $\gamma_{s,s+1}=\ell^2$. If not, $(s+1)^2-4s=(s-1)^2$ is a non-zero square modulo $\ell$ and then $\gamma_{s,s+1}=\ell^2+\ell$.
	\end{proof}
	\begin{theorem} \label{asymptotic probability}
		{Let $S_{\ell}= \{p\;\text{prime} : p\equiv 1\pmod\ell\}$ and $T_{\ell}= \{p\;\text{prime} : p\not\equiv 1\pmod \ell\}$. Then,
		\[
		\lim_{\substack{p\to\infty}, p\in S_{\ell}}\mu_{\Z_p^5}(W_{\ell,p})=\frac{1}{\ell-1}-\frac{1}{\ell^2-1},\;\;\;\;\; \lim_{\substack{p\to\infty}, p\in T_{\ell}}\mu_{\Z_p^5}(W_{\ell,p})=\frac{1}{\ell-1}.
		\]}
	\end{theorem}
	\begin{proof}
		{Both cases follow from \eqref{eq:limgamma} and Lemma \ref{lemma:countmat}.}
	\end{proof}
 Notice in particular that this agrees with Theorem \ref{thm:main}.
 
\begin{remark}\label{rmk:l-torsion}
 We briefly describe an algorithm to compute the probability that an elliptic curve over $\Q_p$ has a non-trivial $\Q_p$-rational $\ell$-torsion point for $p,\ell\geq 5$, $p\neq \ell$. The measure of Weierstrass equations with singular reduction modulo $p$ and a $\Q_p$-rational $\ell$-torsion point can be computed in the exact same way as in the case $\ell=3$ treated in Section \ref{sub:bad-reduction}. To compute the measure of Weierstrass equations with non-singular reduction modulo $p$, we need to generalize the result of Lemma \ref{good reduction count} to $\ell\geq 3$. To make this result explicit, we need to know the {numbers of $\F_p$-rational points on the two-dimensional moduli spaces $X_{\omega}(\ell)$ and $X_{\omega,1}(\ell)$, which are not explicitly known in general}. Following the notation of Section \ref{moduli space section}, {we denote by $X_{\omega,1}(\ell)$ and $X_{\omega}(\ell)$ the moduli spaces of elliptic curves equipped with a nonzero regular differential and, respectively, a nonzero $\ell$-torsion point or a symplectic level-$\ell$ structure. Their numbers of $\F_p$-rational points are denoted by $\#(X_{\omega,1}(\ell)(\F_p))$ and $\#(X_{\omega}(\ell)(\F_p))$, respectively. We first restrict to minimal Weierstrass equations.}

 Assume $p\equiv 1\pmod \ell$.
 \begin{itemize}
 \item If $E$ has split multiplicative reduction, then $\ell\mid p-1=\#\Tilde{E}_{ns}(\F_p)$ and then $E$ always has a $\Q_p$-rational non-trivial $\ell$-torsion point. So the measure is given by the probability of having split multiplicative reduction, that is $(p-1)/2p^2$.
 \item If $E$ has non-split multiplicative reduction, then $\ell\nmid p+1=\#\Tilde{E}_{ns}(\F_p)$ and then $E$ does not have a $\Q_p$-rational non-trivial $\ell$-torsion point. Recall that the multiplication by $\ell$ is an isomorphism over $E_1(\Q_p)$ and $E(\Q_p)/E_0(\Q_p)$ has order at most $4$.
 \item If $E$ has additive reduction, then $\ell\nmid p=\#\Tilde{E}_{ns}(\F_p)$. Moreover, $\ell$ does not divide {the order of} $E(\Q_p)/E_0(\Q_p)$ and so $E$ does not have a $\Q_p$-rational non-trivial $\ell$-torsion point.
 \item If $E$ has good reduction, then there are $\#(X_{\omega}(\ell)(\F_p))/{\ell(\ell^2-1)}$ Weierstrass {equations} in short {form} with full torsion and $\#(X_{\omega,1}(\ell)(\F_p))/{(\ell-1)}-\#(X_{\omega}(\ell)(\F_p))/{\ell(\ell-1)}$ with $E[\ell](\F_p)\cong (\Z/\ell\Z)$.
 \end{itemize}
 In conclusion, 
 \[
 {\mu(W_{\ell,p})}=\frac{p^{10}}{p^{10}-1}\left(\frac{\#(X_{\omega,1}(\ell)(\F_p))}{{(\ell-1)}p^2}-\frac{\#(X_{\omega}(\ell)(\F_p))}{{(\ell^2-1)}p^2}+\frac{p-1}{2p^2}\right).
 \]
 Assume $p\not\equiv 1,-1\pmod \ell$, so we can never have full torsion.
 \begin{itemize}
 \item If $E$ has split multiplicative reduction, then $\ell\nmid p-1=\#\Tilde{E}_{ns}(\F_p)$. So, $E$ has a $\Q_p$-rational non-trivial $\ell$-torsion point {if and only if the reduction} type is $I_{\ell k}$ for some $k$. The measure is then equal to $(p-1)^2/(2p^{\ell+2}(1-1/p^\ell))$.
 \item If $E$ has non-split multiplicative reduction, then $\ell\nmid p+1=\#\Tilde{E}_{ns}(\F_p)$ and then $E$ does not have a $\Q_p$-rational non-trivial $\ell$-torsion point.
 \item If $E$ has additive reduction, then $\ell\nmid p=\#\Tilde{E}_{ns}(\F_p)$. Moreover, $\ell$ does not divide {the order of} $E(\Q_p)/E_0(\Q_p)$ and so $E$ does not have a $\Q_p$-rational non-trivial $\ell$-torsion point.
 \item If $E$ has good reduction, then there are $\#(X_{\omega,1}(\ell)(\F_p))/{(\ell-1)}$ Weierstrass equations in short form with $E[\ell](\F_p)=(\Z/\ell\Z)$.
 \end{itemize}
In conclusion
 \[
 {\mu(W_{\ell,p})}=\frac{p^{10}}{p^{10}-1}\left(\frac{\#(X_{\omega,1}(\ell)(\F_p))}{{(\ell-1)}p^2}+\frac{(p-1)^2}{2p^2(p^{\ell}-1)}\right).
 \]
Assume $p\equiv -1\pmod \ell$. With a similar strategy, we can prove that
 \[
 {\mu(W_{\ell,p})}=\frac{p^{10}}{p^{10}-1}\left(\frac{\#(X_{\omega,1}(\ell)(\F_p))}{{(\ell-1)}p^2}+\frac{(p-1)^2}{2p^2(p^{\ell}-1)}+\frac{p-1}{2p^2}\right).
 \] 
In particular, these formulas imply that {$\mu(W_{\ell,p})$} is a rational number, a fact that is not a priori obvious. We remark that using the results of Section 6 and \cite[Theorem 1]{pannekoek2012p} to handle the case of additive reduction, the same strategy can also be used to prove a similar result in the case $\ell=p$. We omit the details.
 \end{remark}

 \subsection{\texorpdfstring{$\ell$}{}-torsion in families of twists}\label{subsec:twists-l-torsion}
Let $j\in \Q_p$. We want to compute the probability that a Weierstrass equation with $j$-invariant equal to $j$ has a $\Q_p$-rational non-trivial $\ell$-torsion point. This problem is in some sense in between the problem of finding $\ell$-torsion and $\ell$-isogenies. Indeed, as we will show in Corollary \ref{cor:twist}, the probability of having a non-trivial $3$-torsion point in the family of Weierstrass equations with the same $j$-invariant depends only on how many $\Q_p$-rational $3$-isogenies {an element of the family admits}. Notice that this problem is quite easy over $\Q_p$ due to the fact that $\Q_p^*/(\Q_p^*)^2$ has only $4$ elements and so elliptic curves with the same $j$-invariant can be divided in $4$ classes, and elements in the same class are isomorphic over $\Q_p$.
\begin{lemma}
Let $\ell\neq p$ be two odd primes with $p\geq 5$.
 Let $j\in \Q_p{\setminus \{0,1728\}}$ and $E$ be an elliptic curve with $j$-invariant equal to $j$. Let $m$ be the cardinality of the $\Q_p$-rational roots of $\psi_\ell(x)$, the $\ell$-division {polynomial} of $E$.
 Let $P_{j,\ell}$ be the probability that an elliptic curve with $j$-invariant $j$ has a $\Q_p$-rational non-trivial $\ell$-torsion point. Then
 \[
 P_{j,\ell}=\begin{cases}
 0 \text{ if }m=0,\\
 \frac{1}{4} \text{ if } m=\frac{\ell-1}{2},\\
 {\frac{1}{4}\text{ if } m>\frac{\ell-1}{2} \text{ and } p\equiv 1\pmod{\ell},}\\
 {\frac{1}{2}\text{ if } m>\frac{\ell-1}{2} \text{ and } p\equiv -1\pmod{\ell}.}
 \end{cases}
 \]
\end{lemma}
{
\begin{proof}
Let $E : y^2=x^3+ax+b$ be an elliptic curve defined over $\Q_p$ with $j$-invariant equal to $j$. All such elliptic curves are defined by  $y^2=x^3+au^2x+bu^3$ for $u\in \{1,\nu,p,\nu p\}$ where $\nu$ is a non-square in $\Z_p^*$. We give these four $\Q_p$-isomorphism classes equal probability. The isomorphism over $\overline{\Q_p}$ from $E_u$ to $E_{u'}$ is given by the map $(x,y)\to ( (u'/u)x,(u'/u)\sqrt{u'/u}y)$, so $x(P)$ is $\Q_p$-rational if and only if it is $\Q_p$-rational after the twist. We now split into the cases according to the number of roots of $\psi_{\ell}$ defined $\Q_p$. Assume first that there aren't any. Then every non-zero $\ell$-torsion point of $E$ is of the form $(x(P),y(P))$ for $x(P)\notin \Q_p$. Hence, each non-zero $\ell$-torsion point of $E_u$ is of the form $(ux(P),u\sqrt{u}y(P))$ for $x(P)\notin \Q_p$. Thus, $E_u$ does not have a $\Q_p$-rational non-trivial $\ell$-torsion point for all $u$.

Assume now that $\psi_\ell(x(P))=0$ with $x(P)\in \Q_p$ for a point $P$. Then $\psi_\ell$ has at least $(\ell-1)/2$ $\Q_p$-rational zeros (given by $x(nP)$ for $n=1,\dots,(\ell-1)/2$). Assume that there are exactly $(\ell-1)/2$ $\Q_p$-rational zeros. Moreover, $y(P)\in\Q_p(\sqrt{v})$ with $v\in \{1,\nu,p,\nu p\}$ representing the squareclass of $y(P)^2$, and $y(iP)\in \Q_p(\sqrt{v})$ for all $i$. The Galois action on $P$, and hence on its multiples, is given by the character associated with $\Q_p(\sqrt{v})/\Q_p$, so $y(iP)\sqrt{v}\in\Q_p$. Thus, the twisted image of $P$ is a $\Q_p$-rational non-trivial $\ell$-torsion point for $E_v$, and its image on $E_u$ is not $\Q_p$-rational for $u\neq v$. Since all the $\Q_p$-rational roots of $\psi_\ell$ come from multiples of $P$, no other twist has a $\Q_p$-rational non-trivial $\ell$-torsion point. In this case, we can easily conclude that $P_{j,\ell}=1/4$.

Assume now that $\psi_\ell$ has more than $(\ell-1)/2$ $\Q_p$-rational zeros. Then we have two independent elements of $E[\ell](\overline{\Q_p})$ having rational $x$-coordinates. On each of the corresponding lines, the Galois action is given by a character taking values in $\{\pm1\}$. By the Galois equivariance of the Weil pairing, we have $E[\ell](\overline{\Q_p})\cong\chi_u\oplus\chi_{cycl,\ell}\chi_u$, where $\chi_u$ is the quadratic character associated with some $u\in\{1,\nu,p,\nu p\}$. Hence, $\Q_p(E[\ell])\subset\Q_p(\sqrt{\nu},\sqrt{p})$, which implies $\Q_p(\zeta_\ell)\subset\Q_p(\sqrt{\nu},\sqrt{p})$. Since $\Q_p(\zeta_\ell)/\Q_p$ is unramified (as $\ell\neq p$), we conclude that $\zeta_\ell\in\Q_p(\sqrt{\nu})$. Since $[\Q_p(\zeta_\ell):\Q_p]=\ord_\ell(p)$, where $\ord_\ell(p)$ denotes the order of $p$ in $\F_\ell^*$, it follows that $p\equiv\pm1\pmod\ell$.

When $p\equiv1\pmod\ell$, the character $\chi_{cycl,\ell}$ is trivial. Using that for any $v\in\{1,\nu,p,\nu p\}$ we have $E_v[\ell](\overline{\Q_p})\cong E[\ell](\overline{\Q_p})\otimes\chi_v=\chi_{uv}\oplus\chi_{cycl,\ell}\chi_{uv}=\chi_{uv}\oplus\chi_{uv}$, we observe that the twist $E_v$ has a non-trivial $\Q_p$-rational $\ell$-torsion point if and only if $v=u$, and we conclude that $P_{j,\ell}=1/4$.

If $p\equiv-1\pmod\ell$, $\chi_{cycl,\ell}$ is the non-trivial unramified quadratic character. As in the previous case, $E_v[\ell](\overline{\Q_p})\cong E[\ell](\overline{\Q_p})\otimes\chi_v=\chi_{uv}\oplus\chi_{cycl,\ell}\chi_{uv}$, and now there are exactly two different values of $v\in\{1,\nu,p,\nu p\}$ for which one of the characters $\chi_{uv}$ and $\chi_{cycl,\ell}\chi_{uv}$ is trivial. Thus, in this case, $P_{j,\ell}=1/2$.
\end{proof}
}

 \begin{corollary}\label{cor:twist}
 Let $p\geq 5$. Let $j\in \Q_p{\setminus \{0,1728\}}$ and $E$ be an elliptic curve with $j$-invariant equal to $j$. Let $n$ be the cardinality of $\Q_p$-rational $3$-isogenies of $E$. Then,
 \[
 P_{j,3}=\begin{cases}
 0 \text{ if }n=0,\\
 \frac 14 \text{ if } n=1,\\
 \frac 12 \text{ if } n=2,\\
 \frac 14\text{ if }n=4.
 \end{cases}
 \]
 \end{corollary} 
 \begin{proof}
 Notice that if $x(P)$ is a $\Q_p$-rational root of $\psi_3$ for $P\in E[3](\overline{\Q_p})$, then the isogeny with kernel $\{P,2P,0_E\}$ is $\Q_p$-rational. Assume that $n=0$. Then $\psi_3$ does not have any $\Q_p$-rational root, since otherwise $E$ would have a $\Q_p$-rational isogeny. So, $P_{j,3}=0$. If $n=1$, then $\psi_3$ has exactly one $\Q_p$-rational root (given by $x(P)$ with $P$ in the kernel of the isogeny). So, $P_{j,3}=1/4$. 

 It remains to consider the case $n>1$. If $p\equiv 1\pmod 3$, then $E[3](\overline{\Q_p})=\chi_u\oplus\chi_u$ for $u \in \{1,\nu ,p,\nu p\}$, $E$ admits $4$ different $\Q_p$-rational isogenies and $\psi_3$ has $4$ $\Q_p$-rational roots. Then, $n=4$ and $P_{j,3}=1/4$. If $p\equiv 2\pmod 3$, then $E[3](\overline{\Q_p})=\chi_u\oplus\chi_v$ for $\chi_u\neq \chi_v$ and $E$ admits $2$ different $\Q_p$-rational isogenies. So, $n=2$ and $P_{j,3}=1/2$.
 \end{proof}

 \bibliographystyle{plain}
 \bibliography{references}
Stevan Gajovi\'{c}\\
{E.T.S. Ingenieros Informáticos, Universidad Politécnica de Madrid, 28660 Madrid, Spain}\\
{Max Planck Institute for Mathematics, Vivatsgasse 7, 53111 Bonn, Germany} \\
{Charles University, Faculty of Mathematics and Physics, Department of Algebra, Sokolov\-ská~83, 186~75 Praha~8, Czech Republic}\\
{Mathematical Institute
of the Serbian Academy of Sciences and Arts National Institute of the Republic of Serbia, Kneza Mihaila 36, Belgrade 11001}\\
\textit{E-mail address}: \url{stevangajovic@gmail.com}\\
 Lazar Radi\v{c}evi\'{c}\\
{Mathematical Institute
of the Serbian Academy of Sciences and Arts National Institute of the Republic of Serbia, Kneza Mihaila 36, Belgrade 11001}\\
{King's College, Department of Mathematics, King's College London
Strand, London WC2R 2LS, United Kingdom}\\
\textit{E-mail address}: \url{lazaradicevic@gmail.com}\\
 Matteo Verzobio\\
{KU Leuven, Celestijnenlaan 200B 3001 Heverlee (Leuven), Belgium}\\
{ISTA, Am Campus 1 3400 Klosterneuburg, Austria}\\
\textit{E-mail address}: \url{matteo.verzobio@gmail.com}
\end{document}